\documentclass[11pt]{amsart}
\usepackage{graphicx} 
\usepackage[T1]{fontenc}
\usepackage[paper=a4paper, left=34mm, right=34mm, top=41mm, bottom=39mm]{geometry}
\usepackage[foot]{amsaddr}

\usepackage[dvipsnames]{xcolor}
\usepackage[utf8]{inputenc}
\usepackage{amssymb}
\usepackage[centertags]{amsmath}
\usepackage{amscd}
\usepackage{amsthm}
\usepackage{enumerate}
\usepackage{mathtools}
\usepackage{mathrsfs}
\usepackage{tikz}
\usepackage{wrapfig}
\usepackage{float}

\usepackage[doi=false,isbn=false,url=false,eprint=false]{biblatex}
\usepackage{multirow}
\usepackage{amsmath}
\usepackage{booktabs}
\usepackage[super]{nth}
\usepackage[english]{babel}
\usepackage{accents}
\usepackage{enumitem}
\usepackage{xfrac}
\usepackage[colorlinks=true,linkcolor=purple,citecolor=ForestGreen]{hyperref}

\usepackage{epsfig}
\usepackage{color}
\usepackage{indentfirst}
\graphicspath{ {Images/} }
\usepackage{afterpage}
\usepackage{comment}
\usepackage[footnotesize]{caption}
\DeclarePairedDelimiter{\abs}{\lvert}{\rvert}
\DeclarePairedDelimiter{\ceil}{\lceil}{\rceil}
\DeclareMathSymbol{\widetildesym}{\mathord}{largesymbols}{"65}

\newcommand{\N}{\mathbb{N}}
\newcommand{\R}{\mathbb{R}}
\newcommand{\C}{\mathbb{C}}
\newcommand{\E}{\mathbb{E}}
\newcommand{\Z}{Z_{n,[w]}}
\newcommand{\Y}{Y_{n,[w]}}
\DeclareMathOperator{\arccosh}{cosh^{-1}}

\makeatletter
\newtheorem*{rep@theorem}{\rep@title}
\newcommand{\newreptheorem}[2]{%
\newenvironment{rep#1}[1]{%
 \def\rep@title{#2 \ref{##1}}%
 \begin{rep@theorem}}%
 {\end{rep@theorem}}}
\makeatother

\makeatletter
\def\blfootnote{\xdef\@thefnmark{}\@footnotetext}
\makeatother

\newtheorem{theorem}{Theorem}[section]
\newtheorem{defi}{Definition}
\newtheorem{prop}[theorem]{Proposition}
\newtheorem{lemma}[theorem]{Lemma}

\newtheorem{claim}{Claim}

\newreptheorem{theorem}{Theorem}
\newreptheorem{lemma}{Lemma}

\title{The length spectrum of random hyperbolic 3-manifolds}
\author{Anna Roig-Sanchis}
\address{Sorbonne Université and Université Paris Cité, CNRS, IMJ-PRG, F-75005 Paris, France}
\date{\today}
\email{\href{mailto:anna.roig-sanchis@imj-prg.fr}{anna.roig-sanchis@imj-prg.fr}}

\begin{document}

\maketitle

\begin{abstract}
We study the length spectrum of a model of random hyperbolic 3-manifolds introduced in \cite{Bram_Jean}. These are compact manifolds with boundary constructed by randomly gluing truncated tetrahedra along their faces. We prove that, as the volume tends to infinity, their length spectrum converge in distribution to a Poisson point process on $\mathbb{R}_{\geq0}$, with computable intensity $\lambda$. 
\end{abstract}

\section{Introduction}

Perelman's proof of Thurston's geometrisation conjecture \cite{Thurston}, together with other important results such as the surface subgroup theorem \cite{SurfGroupThm}, the virtual Haken theorem \cite{HakenAgol} or the ending lamination theorem \cite{MinskyModel, BrockCanaryMinskyModel}, led to remarkable progress in the study of hyperbolic 3-manifolds, the wildest class among all geometric 3-manifolds.

Even so, there are still plenty of open questions in the field. Many of them are about the asymptotic behaviour of geometric invariants of these manifolds, such as the lengths of their geodesics or the spectral gap (see for instance \cite[Conjecture 1.5]{MageeThomas}).

A way of tackling these problems concerning asymptotic, extremal or "typical" behaviours is by using probabilistic techniques. 
This method have been proven to be very successful in areas like graph theory, and more recently also in hyperbolic geometry (see for instance \cite{ParlierHypGeo, LubotzkyHypGeo, BudCurBram, BrooksMakover, MageeNaudPuder, tetra, Mirzakhani, HideMagee, AnantharamanMonk, LipnowskiWright, WuXue}).

In our case, this probabilistic approach consists in the study of random manifolds. The model of random 3-manifolds we consider for this article is the one called \textit{random triangulations} \cite{Bram_Jean}. In this model, a random 3-manifold $M_n$ is constructed by randomly gluing together $n$ truncated tetrahedra along their hexagonal faces, resulting in a compact, oriented 3-manifold with boundary. By Moise \cite{Moise}, one has that, as $n$ tends to infinity, every compact 3-manifold with boundary gets sampled by this model.
Moreover, it turns out that asymptotically almost surely (a.a.s) these manifolds are hyperbolic  \cite{Bram_Jean}, which makes it a suitable model for studying properties of hyperbolic 3-manifolds.

The purpose of this article is to study the typical behaviour of the \textit{primitive length spectrum} -the (multi-)set of lengths of all primitive closed geodesics- of a random hyperbolic 3-manifold $M_n$ under the model of random triangulations.

\subsection{Main result} 
The main result of the article is then the following:
\begin{theorem} \label{Poisson} 
As $n\rightarrow\infty$, the primitive length spectrum of a random compact hyperbolic 3-manifold with boundary $M_n$ converges in distribution to a Poisson point process (PPP), of computable intensity $\lambda$.
\end{theorem}

The theorem gives a concrete characterisation of the behaviour of the length spectrum of these random hyperbolic 3-manifolds $M_n$. Moreover, since the distribution of a Poisson random variable has an explicit formula, and we have an explicit expression for its mean, this result also tells us that it is possible to compute -with the help of a computer- the limits of the probability that a measurable subset $[a,b]\subset \mathbb{R}_{>0}$ contains $k$ points, for any $k\in\N$ and $a,b>0$ fixed. A more precise statement of the theorem with the defined intensity can be found in Section \ref{statement}.

\subsection{Structure and ideas of the proof}
The structure of the proof of Theorem \ref{Poisson} has some common points with the one of the proof of hyperbolicity of the manifolds $M_n$ \cite[Theorem 2.1]{Bram_Jean}. There, they first construct a specific model of non-compact hyperbolic manifolds -made out of a gluing of hyperbolic ideal right-angled octahedra- named $Y_n$, and observe that these can be transformed to the manifolds $M_n$ via Dehn filling. Then, they prove that after this compactification process, the resulting manifolds $M_n$ are still hyperbolic, using as main tools Andreev's theorem \cite{Andreev} and a result of Futer-Purcell-Schleimer \cite{Futer_Purcell}.

The point here is that the geometry of these non-compact random manifolds $Y_n$ is much better understood -by construction of the manifolds and their hyperbolic metric- than the one of the manifolds $M_n$, so it is easier to study their geometric properties. 

Therefore, in order to prove Theorem \ref{Poisson} regarding the length spectrum of $M_n$, we will follow a similar strategy: first, we will prove the result for these manifolds $Y_n$, and then we will see that after the Dehn filling, the result is still true for the compactified manifolds, that is, the $M_n$.  

The main idea behind the proof of this first part is to translate the problem of counting closed geodesics in $Y_n$ to the problem of counting certain cycles in the dual graph of this complex, which is a random 4-regular graph. Then, the asymptotic behaviour of the expected number of these cycles is given by Theorem \ref{cyclesYn}, using the method of moments. To prove this part it is also essential a technical result concerning the growth of the translation length in terms of the word length of the cycles in the graph (Proposition \ref{lgrow}).

On the other hand, the second part of the proof of Theorem \ref{Poisson} comes down to showing two points: the first and principal, is that the length of the curves after the Dehn fillings of the cusps stays roughtly the same (given by Proposition \ref{length_gammas}). For this, we'll rely as well on Andreev's theorem \cite{Andreev} and Futer-Purcell-Schleimer \cite{Futer_Purcell}, although in a different way. 
And secondly, that closed geodesics don't collapse after the compactification (given by Lemma \ref{homotopy}). Both results -as well as some other statements in the article- are proved to hold asymptotically almost surely (a.a.s), which means with probability tending to 1 as $n\rightarrow\infty$.

Finally, we note that one could also try to carry out a proof for this result using 
the Chen–Stein method for Poisson approximation \cite{PoisAprox, SteinMethod} -instead of the method of moments- which might provide explicit errors terms in the approximation.

\subsection{Organisation}
In Section 2, we recall the probabilistic model of random 3-manifolds introduced in \cite{Bram_Jean}. Then, Section 3 is devoted to the proof of the first step of Theorem \ref{Poisson}. Thus, first, we present the model of hyperbolic manifolds $Y_n$ for which we show the Poisson convergence, and then, we enter into the proof of this intermediary result. Finally, Section 4 contains the proof of the second step of Theorem \ref{Poisson}. We explain first how one passes from the model $Y_n$ to the model $M_n$, and then we continue with the proof of the two points mentioned in the previous section, finishing with the proof of Theorem \ref{Poisson}.

\subsection{Notes and references}
Several models of random hyperbolic surfaces were developed by Mirzakhani \cite{Mirzakhani}, Brooks-Makover \cite{BrooksMakover}, Guth-Parlier-Youg \cite{ParlierHypGeo}, Budzinski-Curien-Petri \cite{BudCurBram} and Magee-Naud-Puder \cite{MageeNaudPuder}. In the case of 3 dimensions, there are two other well known models, called random Heegaard splittings and random mapping tori, that were both introduced by Dunfield and Thurston \cite{tetra}.

The former model is probably the most studied one, and is obtained from a Heegaard splitting, that is, by gluing together two copies of a handlebody $H_g$ of genus $g$ along a "random" orientation preserving diffeomorphism $f$ of the boundary $\partial H_g = \Sigma$. Since the manifold depends only on the isotopy class of $f$, it is well defined for the mapping class $[f]\in \ \textrm{Mod}(\Sigma)$. Thus, a random Heegard splitting is one such that the mapping class is taken at random by doing a random walk in the mapping class group Mod$(\Sigma)$. Note, that this model samples only manifolds of bounded Heegaard genus.
Considerable work has been done in the study of geometric invariants of this model, such as the spectrum of the Laplacian \cite{UrsulaLap}, the growth of the diameter and injectivity radius \cite{FellerDiam}, or that of their volume \cite{GabrieleVol} -providing an answer to the volume conjecture of Dunfield and Thurston \cite[Conjecture 2.11]{tetra}-. These often used combinatorial models similar to those that are behind the proof of the ending lamination conjecture \cite{MinskyModel, BrockCanaryMinskyModel}.

\subsection{Acknowledgements} I would like to thank first my PhD advisor Bram Petri, for his help and patience when discussing the details of this article, and his time spent reading its multiple earlier drafts. I would also like to thank my lab colleagues Pietro Mesquita Piccione and Joaquín Lejtreger for very useful discussions.

\section{Random manifolds: The probabilistic model} \label{modelM_n}

In this section, we explain the model of random hyperbolic 3-manifolds used for our results in this paper. This probabilistic model, called \textit{random triangulations}, is an analogue in three dimensions of Brooks and Makover's model for random surfaces \cite{BrooksMakover}. 

The general idea is to construct manifolds by randomly gluing polyhedra together along their faces. A priori, it seems that the natural choice for creating this complex would be the 3-simplex, i.e., the tetrahedron. However, as shown in \cite[Proposition 2.8]{tetra}, as the number of tetrahedra tends to infinity, the probability that the resulting complex is a manifold tends to 0. This is due to the fact that the neighbourhoods of the vertices of such complexes are not typically homeomorphic to $\R^3$. 

This problem can be solved, however, by truncating the tetrahedra at their vertices (see Figure \ref{truncated1}). Then, the complex we obtain by gluing $n$ of these polytopes along their hexagonal faces, namely $N_n$, is a compact 3-manifold with boundary.

The random aspect of the construction comes from the gluing. The $4n$ hexagonal faces are partitioned into pairs uniformly at random, and for each pair, one of the three cyclic-order-reversing gluings is also chosen with respect to this measure.

\begin{figure}[H]
    \centering
    \includegraphics[scale = 0.4]{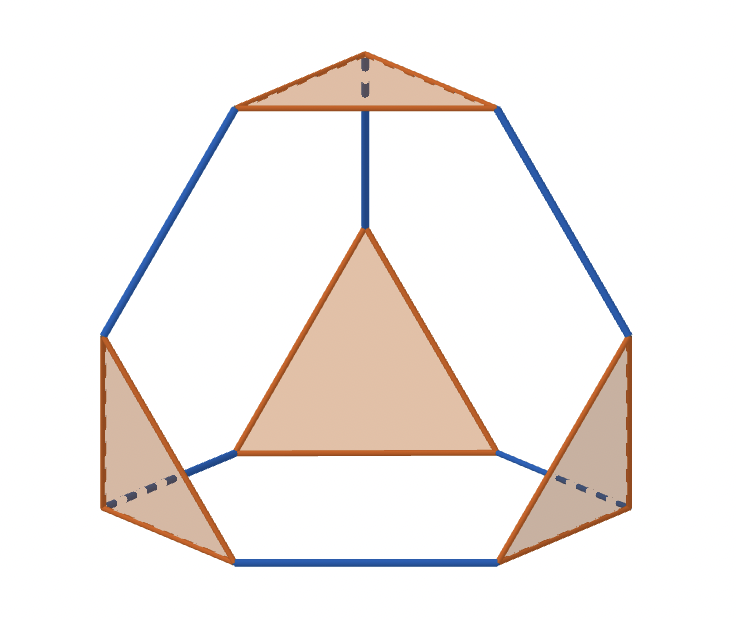} 
    \caption{A truncated tetrahedron. The orange faces are the boundary faces, and the white ones the interior faces. $N_n$ is obtained by gluing these polyhedra along the interior (hexagonal) faces.}
    \label{truncated1}
\end{figure}

One can deduce from classical work by Moise \cite{Moise} that every compact  orientable 3-manifold with boundary can be obtained with this construction. In other words, every such manifold gets, as $n$ goes to infinity, sampled by this model.

Furthermore, we can obtain random closed manifolds from $N_n$: it suffices to take two disjoint copies of $N_n$ and glue their boundaries together using the identity map. Note, however, that these will have additional symmetries, hence do not represent "typical" closed 3-manifolds. The resulting random manifold will be denoted by $DN_n$. 

The dual graph of the complex $N_n$, that is, the graph obtained by considering a vertex in each tetrahedron of $N_n$ and joining them with an edge whenever they have a face in common, is a random 4-regular graph. The construction of this graph follows exactly the one described by the configuration model -a well-known model for random graphs- which consists generally in these two following steps: 
\begin{enumerate}
    \item We consider n vertices. We assign a degree $k_i$ to each vertex, which is represented as $k_i$ half-edges. 
    \item We choose two half edges uniformly at random and connect them to form an edge. The same procedure is repeated with the remaining ones until all of them are used. 
\end{enumerate}

An important observation is that, in this model, the expected number of loops and multi-edges converge to Poisson random variables with parameters $\frac{3}{2}$ and $\frac{9}{4}$ respectively. Thus, for any property P that holds a.a.s. for $N_n$ (i.e, $\mathbb{P}[ N_n \ \textrm{has} \ P] \rightarrow 1$ as $n\rightarrow \infty$), we have:
\begin{align*}
    \mathbb{P}[ N_n \ \textrm{doesn't have} \ P &| \ \{\textrm{its dual graph is simple}\}] \\ &=
    \frac{\mathbb{P}[ N_n \ \textrm{doesn't have} \ P  \ \textrm{and its dual graph is simple}]}{\mathbb{P}[\textrm{Its dual graph is simple}]}  \\ &\leq
    \frac{\mathbb{P}[ N_n \ \textrm{doesn't have} \ P]}{\mathbb{P}[\textrm{There are no loops or multi-edges}]} \xrightarrow{n\rightarrow \infty} 0. 
\end{align*}

Therefore, we can condition our manifold $N_n$ on not having any loops or multi-edges in the dual graph of the tetrahedral complex. Let $M_n$ denote the manifold satisfying it. 
We finish this section by giving the most important geometric property of these manifolds:
\begin{theorem}[\cite{Bram_Jean}, Theorem 2.1] \label{hyperbolic}
    \[\lim_{n\rightarrow \infty} \mathbb{P}[M_n \ \textrm{carries a hyperbolic metric with totally geodesic boundary}] = 1.\]
\end{theorem}
Note that this condition on the dual graph of $M_n$ is needed for the proof of this theorem. Hence, from now on, \textit{these} are the manifolds we're gonna work with.

Mostow's rigidity theorem tells us that, in this case, this metric is unique up to isometry. This implies, in particular, that geometric invariants of $M_n$ like the length spectrum become topological invariants, and so can be understood from the combinatorics of the gluing. This will be the strategy for the proof of the next theorem in Section \ref{sectionY_n}.

\section{Proof step 1: The length spectrum of $Y_n$} \label{sectionY_n}

In this section, we explain first the random model of non-compact hyperbolic 3-manifolds needed for the first step of the proof. Moreover, we describe how the combinatorics of the model gives us information about its curves. After that, we prove a version of Theorem \ref{Poisson} for these manifolds, which we reformulate in more geometric terms as shown next.

We encode the length spectrum of $Y_n$ by the following counting function: for any $a, b>0$, we define:
    \[ C_{[a,b]}(Y_n) \coloneqq \# \{\textrm{primitive closed geodesics of length $\in [a, b]$ on $Y_n$} \},\]
where $C_{[a,b]}(Y_n): \Omega_n \rightarrow \mathbb{R}$ is a random variable, since $Y_n$ is a random manifold.
Then, we prove the following statement, equivalent to the convergence to a PPP.
\begin{theorem} \label{PoissonYn} 
For any finite collection of disjoint intervals $[a_1,b_1], \ldots, [a_t,b_t] \subset \R_{\geq 0}$ , the random vector $(C_{[a_1,b_1]}(Y_n), \ldots, C_{[a_t,b_t]}(Y_n))$
converges jointly in distribution, as $n\rightarrow\infty$, to a vector of independent random variables $$(\mathcal{C}_{[a_1,b_1]}, \ldots, \mathcal{C}_{[a_t,b_t]}),$$ where $\forall i=1,\ldots,t$, $\mathcal{C}_{[a_i,b_i]}$ is Poisson distributed with parameter $\lambda = \lambda([a_i,b_i]) > 0$.

\end{theorem} 

\subsection{Model for hyperbolic 3-manifolds} \label{gluing_oct}
Here we describe the details of the construction presented in \cite[Section 3.1]{Bram_Jean}, and how it is related to the one of Section 1.

The building block for the topological random model $M_n$ was a truncated tetrahedron. Now, observe that if we contract the edges joining the triangular faces of a truncated tetrahedron, we get an octahedron (see Figures \ref{truncated} and \ref{oct}). This will be the building block for $Y_n$.

\begin{figure}[H]
    \centering
    \begin{minipage}{0.45\textwidth}
        \centering
        \includegraphics[width=0.9\textwidth]{Truncated_tetra.png} 
        \caption{A truncated tetrahedron.}
        \label{truncated}
    \end{minipage}\hfill
    \begin{minipage}{0.50\textwidth}
        \centering
        \includegraphics[width=0.9\textwidth]{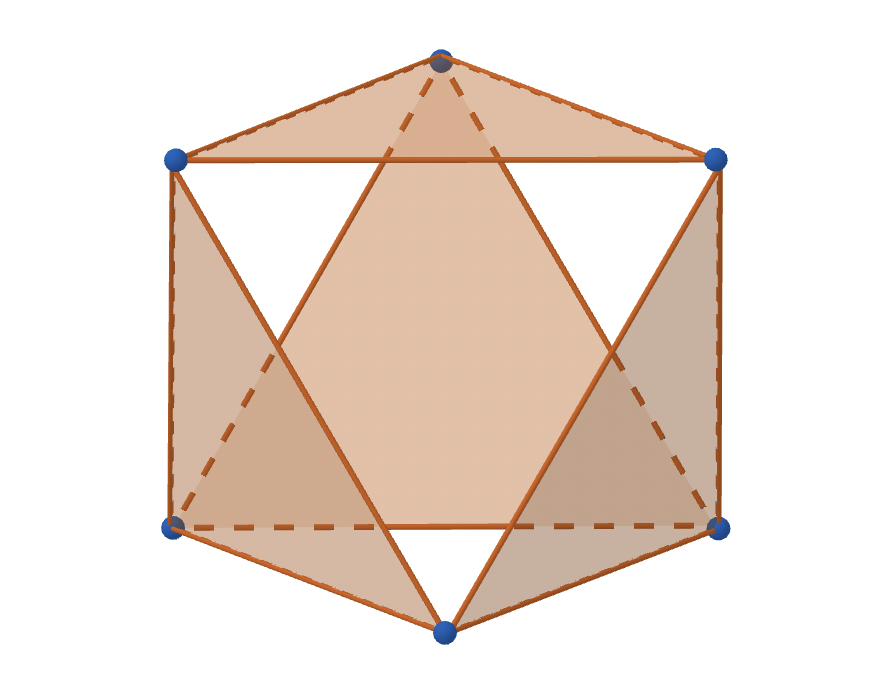}
        \caption{The octahedron resulting from contracting the blue edges of Figure 2}
        \label{oct}
    \end{minipage}
\end{figure}

Now, let's describe the gluing. We consider initially $n\in\N$ copies of a regular octahedron. Then, 
\begin{enumerate}
    \item For each octahedron $O_i$, $i=1, \ldots , n$, we attribute to its vertices a unique label in $\{a^i, b^i, c^i, d^i, e^i, f^i\}$. \\
    We denote the face given by the vertices $v_1, v_2, v_3 \in \{a^i, \ldots, f^i\}$ by a cycle $(v_1 \ v_2 \ v_3)$. The order of the vertices in the cycle determines at the same time an orientation on the face.
    \item We consider four non adjacent faces in each $O_i$, and we partition these $4n$ faces into $2n$ pairs, uniformly at random. We denote this partition by $w_n = (w_n^{(i)})_{i=1}^{2n}$, where $w_n^{(i)} = \{(v_1 \  v_2 \ v_3), (w_1 \ w_2 \ w_3)\}$.
    \item For each pair of faces $w_n^{(i)}$, we choose, again uniformly at random, one of the three cyclic-order-reversing pairings between the vertices. We denote the pairings by $\vartheta_n^{i} = (\vartheta_n^{(i)})_{i=1}^{2n}$, where $\vartheta_n^{(i)} = \{(v_1 \ v_2 \ v_3), (\vartheta_n^{(i)}(v_1)=w_1 \ \vartheta_n^{(i)}(v_2)= w_2 \ \vartheta_N^{(i)}(v_3)= w_3)\}$.
    \item We identify each pair of faces $w_n^{(i)}$ using the pairing of its vertices described by $\vartheta_n^{i}$, for every $i=1,\ldots, 2n$.
\end{enumerate}
The resulting octahedral complex is an oriented compact 3-manifold with boundary. 

Now, when taking out the vertices, this complex admits a hyperbolic metric. Indeed, we can endow each octahedron with the unique (up to isometry) hyperbolic metric of an ideal right-angled regular octahedron. Moreover, we glue them through isometries, so this metric extends nicely to the whole complex. Therefore, after endowing all octahedra in the complex with this hyperbolic metric, we obtain a \textit{complete finite-volume hyperbolic 3-manifold with totally geodesic boundary}.  We will denote this by $X_n$, following the same notation as in \cite{Bram_Jean}.

The dual graph of the complex $X_n$, is, as for $M_n$, a random 4-regular graph. 
In the same way, we can condition $X_n$ on not having loops or bigons in its dual graph. \textit{These} are the manifolds we will work with, and will be denoted by $Y_n$.

The probability space $(\Omega_n, \mathbb{P}_n)$ associated to this model is the following: we define $\Omega_n$ to be the finite set of all possibilities of $w_n$ and $\vartheta_n$, and we choose $\mathbb{P}_n$ to be the uniform probability measure on $\Omega_n$. We have that: 
\[\abs{\Omega_n} = \frac{\binom{4n}{2} \binom{4n-2}{2}\ldots\binom{2}{2}}{(2n)!}\cdot 3^{2n} = (4n-1)!! \ 3^{2n} \]
and so the probability of having one specific configuration is: 
\[\mathbb{P}(\{\textrm{a certain partition $w_n$ and pairing $\vartheta_n$}\}) = \frac{1}{\abs{\Omega_n}}. \]

\subsection{Geometry of $Y_n$}
By construction, the manifold $Y_n$ is hyperbolic. We record this information in the following lemma:
\begin{lemma}
The manifold $Y_n$ carries a complete hyperbolic metric of finite volume with totally geodesic boundary.
\end{lemma}
As such, we can use the following standard facts from hyperbolic geometry to help us study its length spectrum.

First, we have that every element in the set of free homotopy classes of closed curves in $Y_n$ that is neither trivial or homotopic to a cusp, is represented by a unique closed geodesic \cite[Prop 4.1.13]{Martelli}.

Moreover, as it has totally geodesic boundary, we can write $Y_n$ as a quotient $\mathcal{C}(\mathbb{H}^3)/\Gamma$, where $\mathcal{C}(\mathbb{H}^3)$ is a convex domain of $\mathbb{H}^3$ and $\Gamma$ is a discrete torsion-free subgroup of orientation preserving isometries of $\mathbb{H}^3$. It is known that Isom$^{+}(\mathbb{H}^3) \cong \mathrm{PSL}(2,\C)$, so we can think of $\Gamma$ as a subgroup of matrices in $\mathrm{PSL}(2, \mathbb{C})$.
Then, we have that there is a bijection between the conjugacy classes in $\Gamma$ and the free homotopy classes of closed curves \cite[Section 4.1.5]{Martelli}. This means that, given a closed curve $\gamma$, it corresponds to a conjugacy class of matrices $[M_\gamma]$ in $\mathrm{PSL}(2, \mathbb{C})$. 

An important thing about these fact is that from $[M_\gamma]$ one can compute the length of the corresponding geodesic in the homotopy class of $\gamma$. This length is exactly the translation length of $[M_\gamma]$, that is, the distance between $p$ and $M_\gamma(p)$ for any $p\in \textrm{axis}(M_\gamma$) - the geodesic line in $\mathbb{H}^3$ preserved by the isometry -. This is given by: 
\begin{equation} \label{l}
    l_{\gamma}(M_{\gamma}) = 2\mathrm{Re}\bigg[\arccosh \bigg(\frac{\textrm{trace}([M_{\gamma}])}{2}\bigg)\bigg].
\end{equation}

The first goal is then, to try to describe precisely this class $[M_\gamma]$ of $\Gamma$. For that, we will look at the dual graph of $Y_n$.

\subsection{Curves and paths}
The dual graph of $Y_n$, that will be denoted by $G_{Y_n}$, encodes part of the combinatorics of the complex. To completely determine it, we need to include in $G_{Y_n}$ the orientation-reversing gluings of the pair of faces corresponding to the pair of half-edges. Then, from this "enriched" graph, we can get information about the length of the curves of the manifold $Y_n$. Ultimately, we will see that the distribution of the number of closed curves of a fixed length in $Y_n$ can be studied by looking at the distribution of the number of certain closed paths in $G_{Y_n}$.

Let us start by showing how curves in $Y_n$ and paths in $G_{Y_n}$ are related. A first and essential observation is that any curve in $Y_n$ can be homotoped to a path on the dual graph. This homotopy can be done as follows: we cut the curve into pieces, each one corresponding to the part of the curve that enters and leaves exactly once some octahedron along one of the four non-adjacent faces. Then, if this part of the curve enters and leaves through the same face, we homotope it to the middle point of the face. Otherwise, we homotope the entry and exit points to the center of the faces, and all the remaining part of the curve to the graph. Figure \ref{Homotopy_path} shows an example.

Once this process is done, we remove the possible backtracking from the graph, i.e, we remove the edges through which the path goes and turns back in the opposite direction after reaching some vertex. The resulting path is what we call the \textit{reduced path}. From now on, these are the paths we will consider.

\begin{figure}[H] 
    \centering
    \begin{minipage}{0.45\textwidth}
        \centering
        \includegraphics[width=0.9\textwidth]{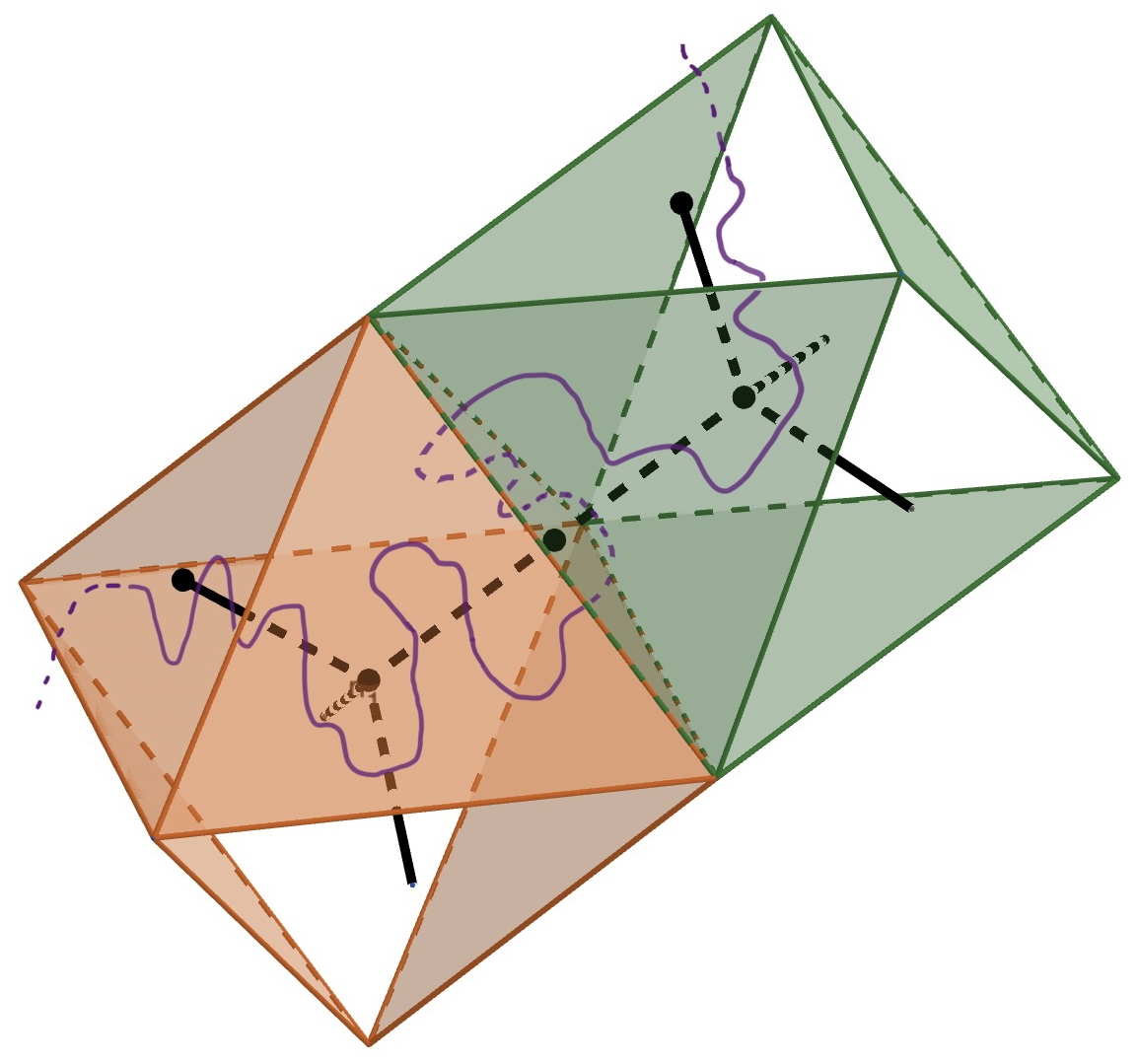} 
    \end{minipage}\hfill
    \begin{minipage}{0.45\textwidth}
        \centering
        \includegraphics[width=0.9\textwidth]{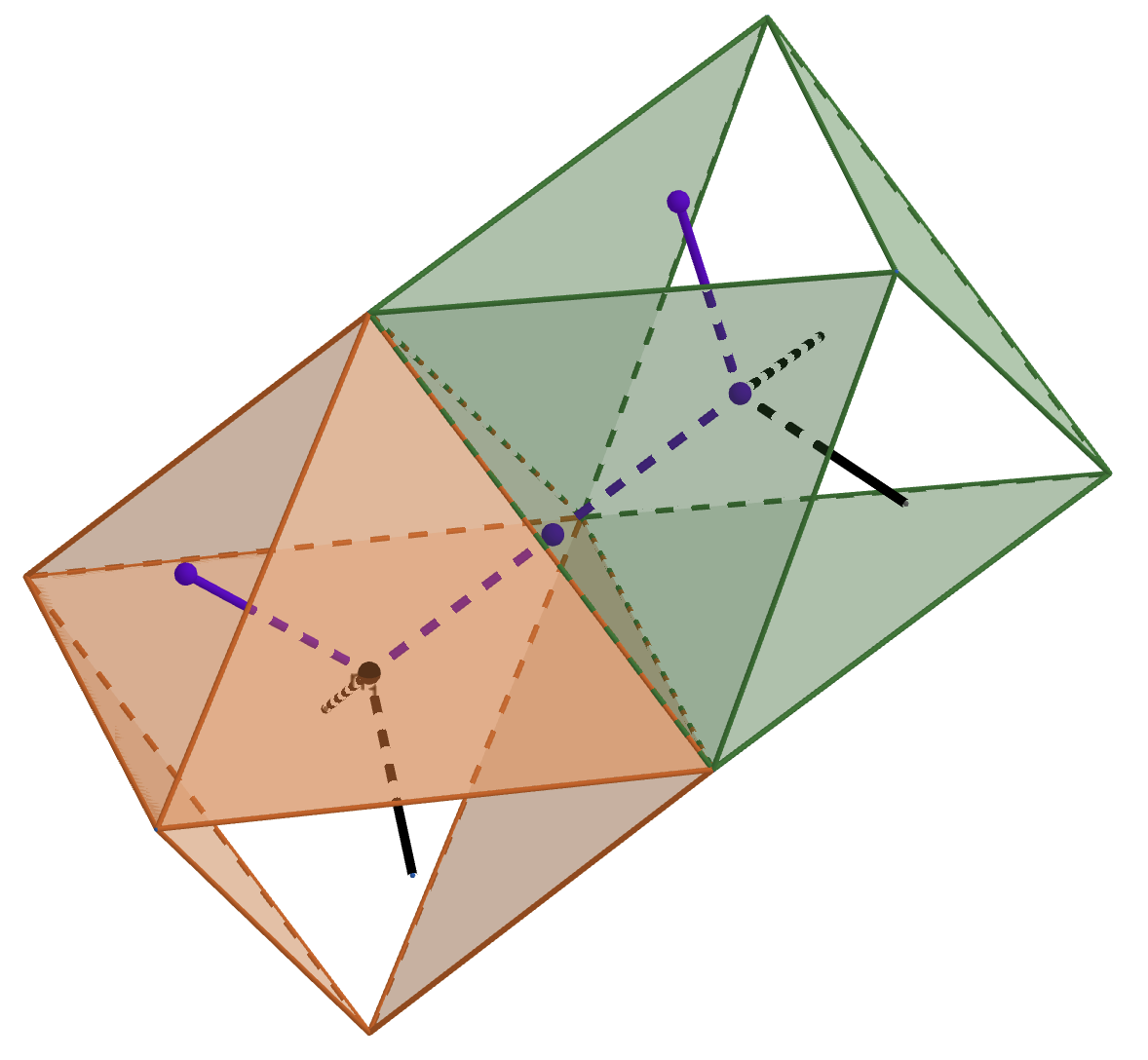}
    \end{minipage}
    \caption{Homotopy of a curve into the dual graph.}
    \label{Homotopy_path}
\end{figure}

Thus, in order to describe a closed curve $\gamma$, we can analyse its reduced path in the dual graph along the octahedra involved.

Now, a closed path of some length $k\in \N $ on this enriched 4-regular graph mention above, can be described by picking a midpoint of an edge as a starting point, and giving a sequence of "movements" $(w_{1}\Theta_1, \ldots, w_{k}\Theta_k)$ that returns to this starting point. Each $w_{i}$ indicates the direction it takes at the $i$-th vertex it encounters, i.e, whether the path is going straight ($S$), turning right ($R$) or turning left ($L$) with respect to the direction of the travel. On the other hand, $\Theta_i$ describes in which of the 3 possible cyclic-order-reversing orientations the two faces of the octahedra are glued.

These movements take on a geometrical meaning when translating our picture to the hyperbolic space. That is, by assigning ideal coordinates in $\mathbb{H}^3$ to the first octahedron, and placing the next ones with the information given by the previous sequence. They, then, correspond to the mapping of one face of an ideal octahedron to some other of its non adjacent faces. These actions are compositions of translations and rotations, and so they are orientation preserving isometries of $\mathbb{H}^3$. As mentioned before, these can be identified with elements of PSL$(2,\C)$.
This group acts on the upper half-space model by Möbius transformations \cite[Theorem 1.8]{Hyp_Kleinian}: 
\[ \begin{bmatrix}
 a & b \\ c & d \end{bmatrix} \cdot z = \frac{az+b}{cz+d}.\]

Therefore, these movements on the graph correspond to some Möbius transformations on the hyperbolic 3-space:
\[S\theta^j, R\theta^j, L\theta^j:\mathbb{H}^3 \rightarrow \mathbb{H}^3, \quad j=0,1,2, \] 
that send a triple of points -realising a face- to another triple in some cyclic order. They are described by these matrices in $PSL(2,\C)$: 
\[ S = \begin{pmatrix} 1 & 1 \\ 0 & 1 \end{pmatrix} \quad \quad R = \begin{pmatrix} -1 & i \\ i-1 & i \end{pmatrix} \quad \quad L = \begin{pmatrix} i & i \\ i+1 & 1 \end{pmatrix} \quad \quad \theta = \begin{pmatrix} 0 & i \\ i & 1 \end{pmatrix}. \]

As there are three possible directions and orientations at each step, we have nine possible isometries:
\begin{align*}
    S\theta^0 &= S = \begin{pmatrix} 1 & 1 \\ 0 & 1 \end{pmatrix} \quad \ \ \quad \quad S\theta = \begin{pmatrix} i & i+1 \\ i & 1 \end{pmatrix} \quad \quad \quad S\theta^2 = \begin{pmatrix} i-1 & i \\ i & 0 \end{pmatrix} \\ 
    R\theta^0 &= R = \begin{pmatrix} -1 & i \\ i-1 & i \end{pmatrix} \quad \quad R\theta = \begin{pmatrix} 1 & 0 \\ 1 & 1 \end{pmatrix} \quad \ \ \ \ \quad \quad R\theta^2 = \begin{pmatrix} 0 & i \\ i & i+1 \end{pmatrix} \\ 
    L\theta^0 &= L = \begin{pmatrix} i & i \\ i+1 & 1 \end{pmatrix} \quad \quad L\theta = \begin{pmatrix} -1 & i-1 \\ -i & i \end{pmatrix} \quad \quad L\theta^2 = \begin{pmatrix} i+1 & 1 \\ 1 & 1-i \end{pmatrix}.
\end{align*}
All these transformations are either parabolic or hyperbolic. More precisely, for each direction, there is one parabolic element - $S$, $R\theta$ and $L\theta^2$ - and two hyperbolic ones.

The product of the elements corresponding to a path (and so a curve $\gamma$) describe the path and it is again an isometry, taking some initial face to the last face the curve goes through. These are the $M_\gamma$ referred to in the previous section.

From now on, we will call \textit{words} these group elements describing curves in the manifold. They will be denoted by:
\[w = w_{1} \Theta_1 \cdot w_{2} \Theta_2 \cdot \ldots \cdot w_{\abs{w}} \Theta_{\abs{w}}\]
where $w_{i} \in \{ S, R, L \}$ and $\Theta_i \in \{Id, \theta , \theta^2\}$ for $i=1, \ldots, \abs{w}$, and $\abs{w}$ is the length of the word, being the number of elements $w_i\Theta_i$ in $w$. For an example, see Figure \ref{Path}.

\begin{figure}[H]
    \centering
    \includegraphics[scale=0.35]{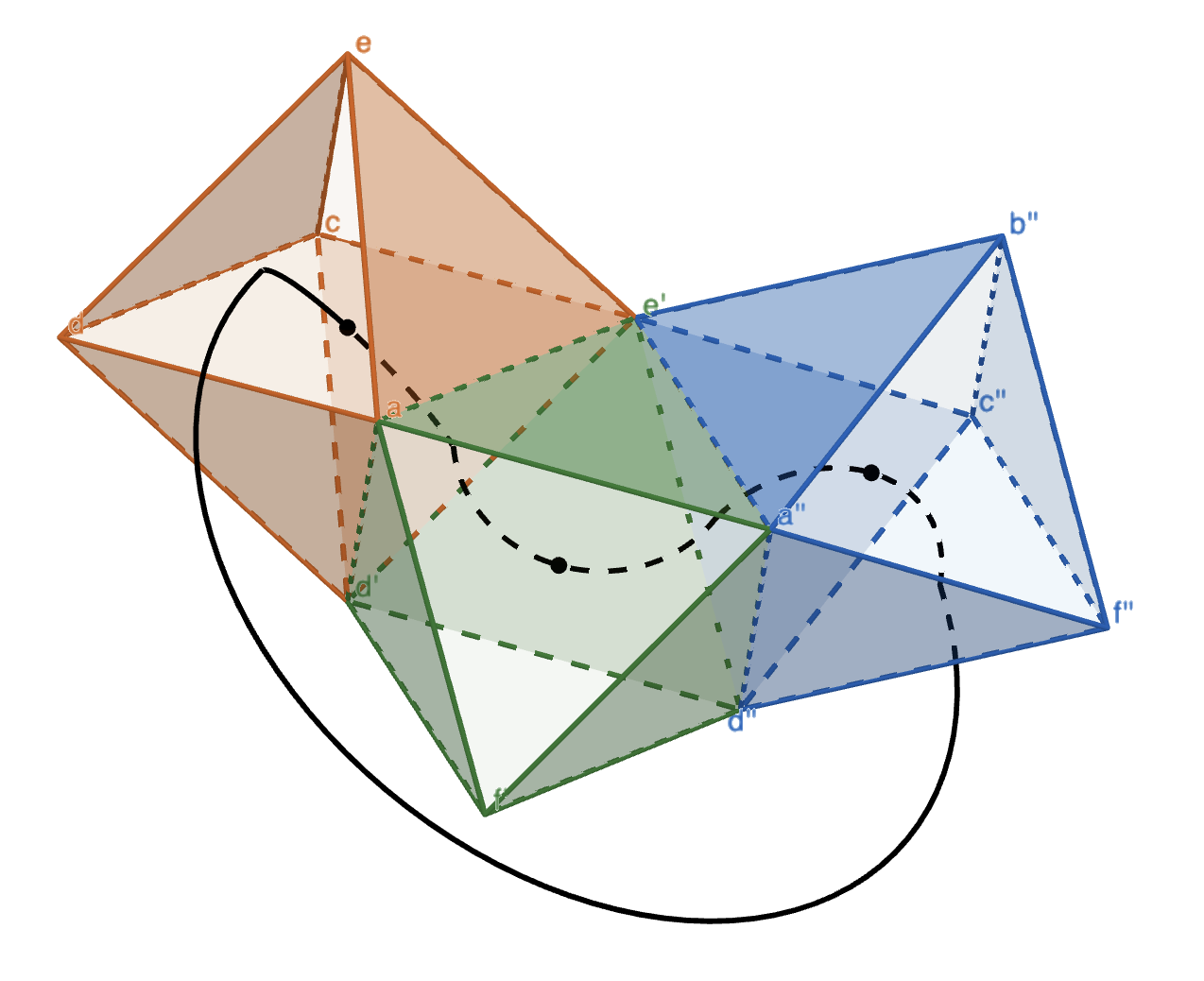}
    \caption{A closed curve in the octahedral complex. If we start at the orange octahedron, an element $w$ describing this curve could be $w = R\theta S L\theta^2$.}
    \label{Path}
\end{figure}

\subsubsection{Words}
In order to compute "the" word corresponding to some path, we need to precise which ideal vertex goes to which. Note that all previous Möbius transformations take the triple of points $(0, i, \infty)$ to some other triple in $\mathbb{H}^3$.
So, in practice, to compute the word, we put every octahedron the curves goes through in the standard position (see Figure \ref{hyp_oct}), and look for the Möbius transformations that sends this triple $(0, i, \infty)$ to the triple of points realising the face of the next octahedron, in a specific order.

\begin{figure}[H]
    \centering
    \begin{minipage}{0.60\textwidth}
        \centering
        \includegraphics[width=0.9\textwidth]{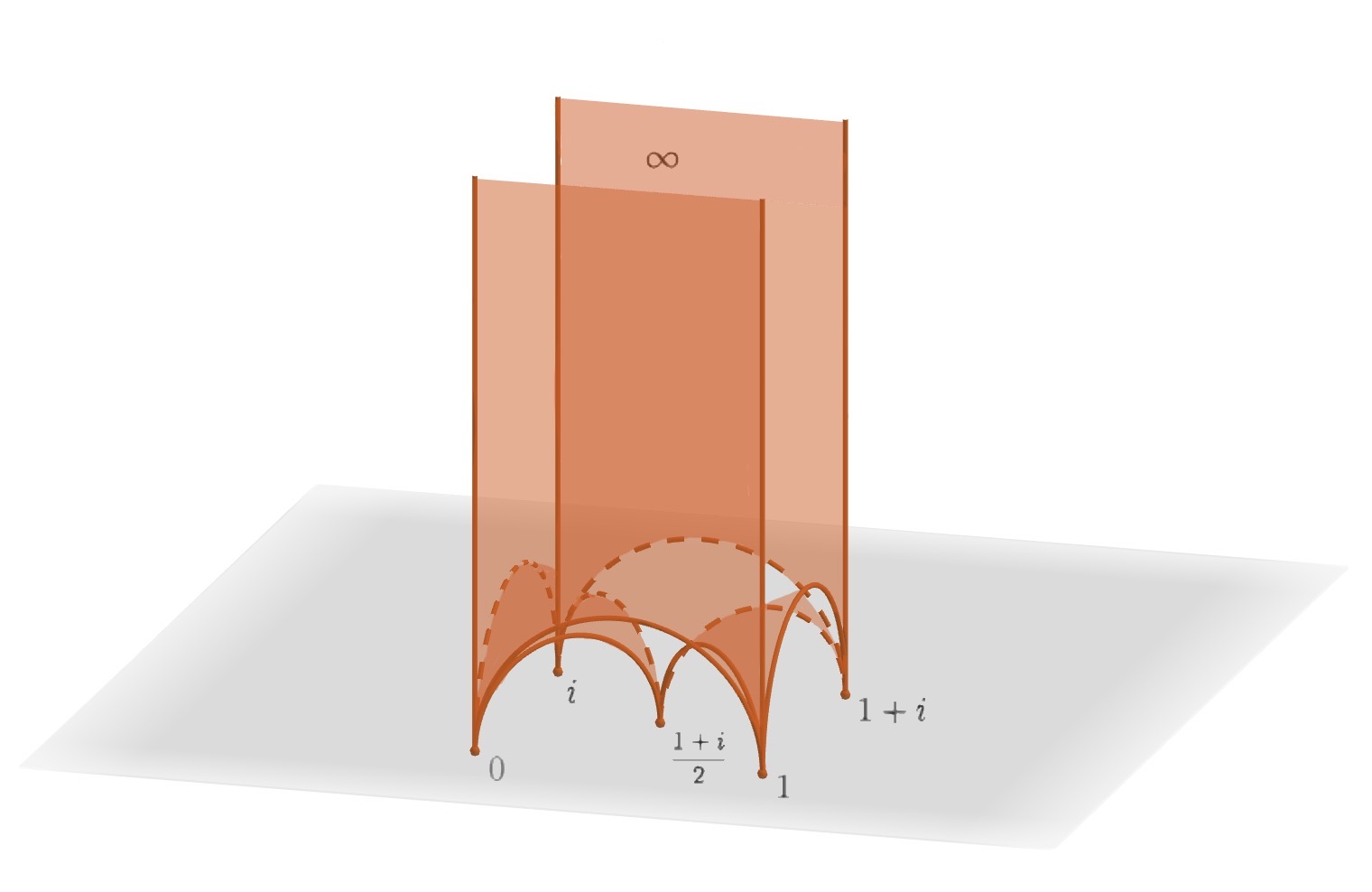} 
    \end{minipage}\hfill
    \begin{minipage}{0.35\textwidth}
        \centering
        \includegraphics[width=0.9\textwidth]{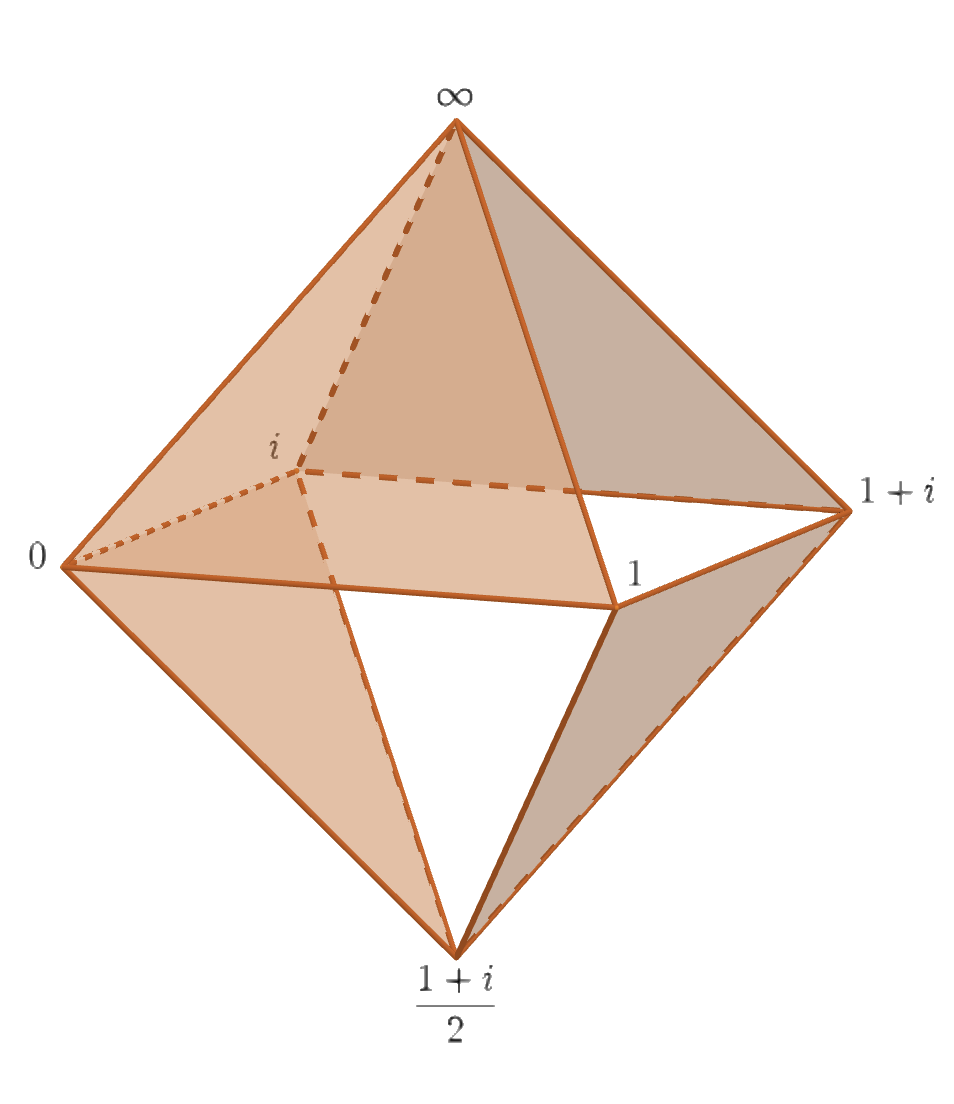}
    \end{minipage}
    \caption{Ideal regular octahedron in the upper half-space model $\mathcal{H}^3$ (left) and its topological image (right).}
    \label{hyp_oct}
\end{figure}

Then, we will call \textit{going straight} when the image of $(0, i, \infty)$ is the tuple $(1, i+1, \infty)$ -transformation carried out by S-, \textit{right} when its image is $(0, \frac{i+1}{2}, 1)$ -realised by R- and \textit{left} when its image is $(i, i+1, \frac{i+1}{2})$ -carried out by L-. On the other hand, the element $\Theta_i$ will precise to whom is each element in $(0, i, \infty)$ sent.

One can observe that, with this procedure, we're changing the way we uniformize $Y_n$ as the quotient of a convex set of $\mathbb{H}^3$ by a Kleinian group $\Gamma$ all the time. Indeed, we're considering  a different conjugate of $\Gamma$ every time we set the next octahedron the curve traverses in the standard position, as we are changing the basepoint at each step. However, this doesn't have influence in the computation of the length of the curve.

Another important observation is that, depending on our choice of starting point and direction, the same curve gives rise to many different words. 
Similarly, depending on the cyclic order we identify the vertices in the tuples when going straight, right or left in every octahedron the curve traverses, we also obtain different words. Thus, we need to define a notion of equivalence class of words.

\vspace{0.2cm}
Let $W$ denote the words $w$ formed with $\{S, R, L\}$ and $\{Id, \theta, \theta^2 \}$.
\begin{itemize}[leftmargin=20px] 
    \item[A)] We focus first on the case when the cyclic order of the ideal coordinates varies: recall that for each incoming face, we can identify its vertices $(v_1 \ v_2 \ v_3)$ with $\{0, i, \infty\}$ in three possible ways -$(0 \ i \ \infty), (\infty \ 0 \ i), (i \ \infty \ 0)$-. This implies that, for a curve going through $k$ octahedra, there will be $3^k$ possible choices to make, each of them yielding a different word, but describing the same curve. 
For every change in the cyclic order of the coordinates in some octahedron $O_i$, the word changes as follows: 
\vspace{0.2cm}

\begin{itemize}

    \item[$\diamond$] The orientation $\Theta_{i-1}$ of the gluing of $O_{i-1}$ and $O_i$ does one twist, i.e, gets multiplied by $\theta$ (here observe that $\theta^3=Id$).
    \vspace{0.15cm}
    
    \item[$\diamond$] The orientation $\Theta_i$ of the gluing of $O_1$ and $O_{i+1}$ does one twist, i.e, gets multiplied by $\theta$.
    \vspace{0.15cm}
    \item[$\diamond$] The direction of the movement in $O_i$ changes following the cyclic order $(R \ S \ L)$.
\end{itemize} 

\vspace{0.3cm}
For example, if consider the word $w = w_1\Theta_1 \cdot w_2\Theta_2 = S \cdot R\theta$, the $3^2=9$ words describing the same curve while changing the cyclic order of the coordinates in $O_1$ and $O_2$ (or equivalently the orientation of the gluings $\Theta_1$ and $\Theta_2$) are: 
\vspace{0.25cm}

\begin{center}
\begin{tabular}{ |c|c|c|c| } 
\hline
 $O_2$ $\backslash$ $O_1$ & $1$ & $2$ & $3$ \\ 
 \hline
 $1$ & $SR\theta$ & $L\theta R\theta^2$ & $R\theta^2 R$ \\ 
 \hline
 $2$ & $S\theta S\theta^2$ & $L\theta^2 S$ & $R S\theta$ \\ 
 \hline
 $3$ & $S\theta^2 L$ & $LL\theta$ & $R\theta L\theta^2$ \\ 
 \hline
\end{tabular}
\end{center}
\vspace{0.2cm}

where $(1) = (0 \ i \ \infty), (2) = (\infty \ 0 \ i)$, and $(3) = (i \ \infty \ 0)$.

\vspace{0.4cm}
\item[B)] Now, in each of the previous words, we are considering the same starting point and direction. Changing any of these parameters will gives us also new words for the same curve. Thus, we will say that two words $w,w`\in W$ are also equivalent if either:
\vspace{0.2cm}

\begin{itemize}
    \item[$\diamond$] $w'$ is a cyclic permutation of $w$ (considering $w_i\Theta_i$ as one element). \\ 
    For instance, if $w = w_1\Theta_1 \cdot w_2\Theta_2 = S R\theta$, there would be one cyclic permutation $w'=R\theta S$. 
    \vspace{0.15cm}
    
    \item[$\diamond$] $w'$ is a cyclic permutation of $w^{\ast}$, where $w^{\ast}$ is the word obtained by reading $w$ backwards, and changing the orientation $\Theta_i$ of each letter by the orientation $\Theta_{i-1}$ of the previous one. In other words, if $w = w_{1} \Theta_1 \cdot w_{2} \Theta_2 \cdot \ldots \cdot w_{\abs{w}} \Theta_{\abs{w}}$, then 
    \[ w^{\ast} =   w_{\abs{w}} \Theta_{\abs{w}-1} \cdot  w_{\abs{w}-1} \Theta_{\abs{w}-2} \cdot \ldots \cdot  w_{2} \Theta_1 \cdot  w_{1} \Theta_{\abs{w}}. \]

    \vspace{0.15cm}
    Following the previous example, the word $w^{\ast}$ of $w = w_1\Theta_1 \cdot w_2\Theta_2 = S R\theta$ is $w^{\ast}= RS\theta$, so the $w'$ given by these transformations would be: $\{RS\theta, S\theta R\}$.
\end{itemize}

\end{itemize}
\vspace{0.15cm}

With all, we consider the following:
\begin{defi}
The equivalence class of a word $w\in W$ is formed by all words $w'\sim w$ resulting from any of the transformations described in A) and B). We denote it by $[w]$.
\end{defi}

The previous formula (\ref{l}) gives us a precise relation between the length of closed geodesics and their (class of) words. 
Thus, in order to study $C_{[a,b]}(Y_n)$, we will need to find the (classes of) words that correspond to (homotopy classes of) closed curves on $Y_n$.
Then, by counting the number of homotopy classes of curves corresponding to each of these $[w]$, we will have insight on the number of closed geodesics of bounded length $\in [a,b]$.

We will see next that this argument can be indeed carried out substituting the counting of homotopy classes of curves by the counting of certain paths in $G_{Y_n}$.

\vspace{0.2cm}

\subsection{Distribution of cycles}

We're interested in counting closed geodesics in $Y_n$. One can see that each of them can be contracted into non-homotopic closed paths in the dual graph. This one-to-one correspondence between free homotopy classes of closed curves in the manifold $Y_n$ and closed paths in $G_{Y_n}$, is due to the fact that the octahedral complex deformation retracts onto the graph. Hence, any two curves in $Y_n$ will be homotopic if and only if their representatives in the dual graph are so.

Therefore, in order to count homotopy classes of curves, we can count closed paths in the dual graph. From now on, we will adopt the following terminology: we will call any closed path a \textit{circuit}, and a simple closed path a \textit{cycle}.

It might happen that some simple closed curves in $Y_n$ are homotopic to non-simple circuits in the graph. However, as it is written later on in the proof of Theorem \ref{cyclesYn}, the number of circuits of bounded length in the graph that are not cycles goes asymptotically almost surely to 0 as $n$ tends to infinity. This tells us, then, that it is enough for our purpose to study the number of cycles in $G_{Y_n}$. 
With this in mind, we consider the random variable: 
\[Z_{n, [w]} : \Omega_n \rightarrow \N, \quad n\in \N, \ [w]\in \mathcal{W} \coloneqq W /\sim,\]

defined as

\[Z_{n, [w]} (\omega) \coloneqq  \#\{\textrm{cycles $\gamma$ on $G_{Y_n}$ : $\gamma$ is described by $[w]$}\}. \]

\vspace{0.2cm}
We prove the following about the asymptotic behaviour of $\Z$:
\vspace{0.15cm}

\begin{theorem}\label{cyclesYn}
Consider a finite set $\mathcal{S}$ of equivalence classes of words in W. Then, as $n \rightarrow \infty$,
\[\Z \rightarrow Z_{[w]} \quad \textrm{in distribution for all} \ [w]\in \mathcal{S},\]
where:
\begin{itemize}
    \item $Z_{[w]}: \N \rightarrow \N$ is a Poisson distributed random variable with mean $\lambda_{[w]}= \frac{\abs{[w]}}{3^{\abs{w}}2\abs{w}}$ for all $[w]\in \mathcal{S}$.
    \item The random variables $Z_{[w]}$ and $Z_{[w']}$ are independent for all $[w],[w'] \in \mathcal{S}$ with $[w] \neq [w']$.
\end{itemize}
\end{theorem}

The proof of this result follows the same structure as the proof of Bollobás' theorem on the asymptotic number of cycles in a random regular graph \cite[Theorem 2]{Bollobas}.
The argument is based on a version of the \textit{method of moments}, which consists on the following: let $X_{n,1}, \ldots, X_{n,k}$ be random variables, for $k\in\N$. We define the variables $(X_{n,i})_m$ as $(X_{n,i})_m = X_{n,i}(X_{n,i} -1) \cdots (X_{n,i}-m+1)$.
Then, if $\lambda_1, \ldots, \lambda_n \geq 0$ are such that, as $n\rightarrow \infty$, 
\[\E[(X_{n,1})_{m_1} \cdots (X_{n,k})_{m_k}] \rightarrow \lambda_{1}^{m_1} \cdots \lambda_{k}^{m_k},\]
for every $m_1, \ldots, m_k \geq 0$, we have that $(X_{n,1}, \ldots, X_{n,k}) \xrightarrow{d} (X_1, \ldots, X_k)$, where the $X_i$ are independent Poisson random variables of parameter $\lambda_i$ \cite[Theorem 6.10]{Moments}.

\begin{proof}
Take as random variables the $\Z$ and consider, for all $n\in \N$ and $[w]\in \mathcal{S}$:
\[(Z_{n,[w]})_{m_{[w]}} = Z_{n,[w]}\cdot (Z_{n,[w]}-1) \cdot \ldots \cdot (Z_{n,[w]} - m_{[w]} + 1),\]
where $m_{[w]}\in \N$.
Then, we prove that $\exists$ $(\lambda_{[w]})_{[w]\in \mathcal{S}}$, with $\lambda_{[w]} \in \R$, such that as $n \to \infty$, 
\[\E\Bigg[\prod_{[w]\in \mathcal{S}} (\Z)_{m_{[w]}} \Bigg] \rightarrow \prod_{[w]\in \mathcal{S}} \lambda_{[w]}^{m_{[w]}} \quad \textrm{for all} \ (m_{[w]})_{[w]\in \mathcal{S}}\in \N ^{\abs{\mathcal{S}}},\]
where $\lambda_{[w]} = \dfrac{\abs{[w]}}{2\abs{w}3^{\abs{w}}}$ for all $[w]\in \mathcal{S}$.

We start with the first moment: $\E[\Z]$. If we denote by $Lab_{[w]}$ the set of all labellings of a $[w]-\textrm{cycle}$, we can write this expectation as: 
\begin{align*}
    \E[\Z] &= \sum_{l\in Lab_{[w]}} \E[\mathbf{1}_{\{\textrm{the labelling $l$ appears in} \ G_{Y_n}\}}]  \\ 
    &= \sum_{l\in Lab_{[w]}} \mathbb{P}[\{\textrm{the labelling $l$ appears in} \ G_{Y_n}\}] = a_{n,[w]}\cdot p_{n,[w]}
\end{align*}
where $a_{n,[w]}$ denotes the number of possible labels of a $[w]$-cycle, and $p_{n,[w]}$ the probability that an element of $\Omega_n$ contains a given set of $\abs{w}$ pairs of half edges. 

To count $a_{n, [w]}$, we will fix first a starting vertex and a direction. Since there are 2 possible directions in a cycle, and $\abs{w}$ possible starting vertices, we will be counting in fact $2\abs{w}a_{n, [w]}$. 

Hence, if the word associated to a $[w]$-cycle is: 
\[w = w_{1} \Theta_1 \cdot w_{2} \Theta_2 \cdot \ldots \cdot w_{\abs{w}} \Theta_{\abs{w}},\]
where $w_i \in \{ S, R, L \}$ and $\Theta_i \in \{ Id, \theta, \theta^2 \}$ for $i=1, \ldots, \abs{w}$, we can describe a directed cycle with starting vertex $v_1$ by a list: 
\[ \{ (x_1, w_1\Theta_1 x_1), (x_2, w_2\Theta_2 x_2), \ldots, (x_{\abs{w}}, w_{\abs{w}}\Theta_{\abs{w}} x_{\abs{w}}) \}, \]
where $x_i$ is a half-edge of $v_i$ and $w_{i}\Theta_i x_1$ is the half edge on the left of $x_i$ if $w_i=L$, on the right of $x_i$ if $w_i=R$, and in front of $x_i$ if $w_i=S$, for every $i=1, \ldots, \abs{w}$ (see Figure \ref{half edges}).

\begin{figure}[H]
    \centering
    \includegraphics[scale=0.3]{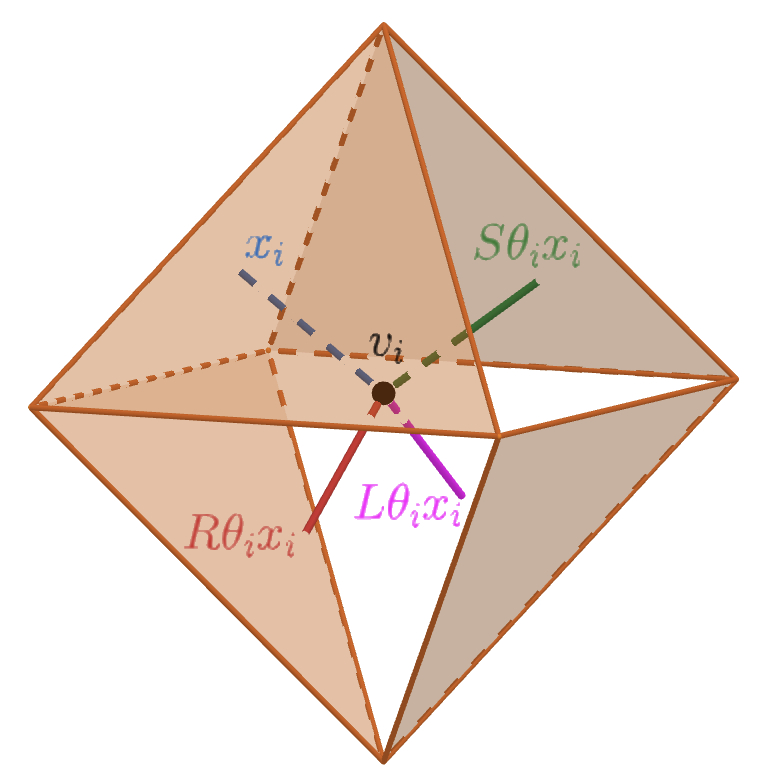}
    \caption{Half edges of $v_i$}
    \label{half edges}
\end{figure}

Taking into account that there are $n$ vertices to choose from (we consider $n$ octahedra), and that in every vertex we have 4 possibilities for $x_i$, 
we have: 
\[2\abs{w}a_{n, w} = 4^{\abs{w}} n (n-1) (n-2) \ldots (n - \abs{w}+1),\]
and since we get these lists for all the representatives of $[w]$, we finally obtain that: 
\[ a_{n, [w]} = \frac{\abs{[w]}}{2\abs{w}}4^{\abs{w}} n (n-1) (n-2) \ldots (n - \abs{w}+1).\]

Now we compute $p_{n, [w]}$.  
We have that the total number of possible partitions of the set of half-edges -of cardinal $4n$- into pairs is:
\[\mathcal{N}(2n) = \frac{\binom{4n}{2} \binom{4n-2}{2}\ldots\binom{2}{2}}{(2n)!} = (4n-1)!!. \]

Moreover, if we fix $k$ independent (vertex disjoint) edges, there are:
\[\mathcal{N}_k(2n) = \frac{\binom{4n-2k}{2} \binom{4n-2k-2}{2}\ldots\binom{2}{2}}{(2n-k)!} = (4n-2k-1)!! \]
configurations containing these k edges.

However, in our case, we also need to consider the 3 possible orientations in every join of two half edges. With this, $p_{n, [w]}$ is given by:

\[ p_{n, [w]} = \frac{3^{2n-\abs{w}}(4n-2\abs{w}-1)!!}{3^{2n}(4n-1)!!} =\frac{1}{3^{\abs{w}}(4n-1)(4n-3)\ldots(4n-2\abs{w}+1)}. \]
All together, it yields:
\begin{align*}
\E[\Z] &= \frac{\abs{[w]}}{2\abs{w}} \frac{4^{\abs{w}} n (n-1) (n-2) \ldots (n - \abs{w}+1)}{3^{\abs{w}}(4n-1)(4n-3)\ldots(4n-2\abs{w}+1)} \\ 
&\sim \frac{\abs{[w]}}{2\abs{w}}\frac{4^{\abs{w}} n^{\abs{w}}}{3^{\abs{w}}(4n)^{\abs{w}}} \longrightarrow \frac{\abs{[w]}}{2\abs{w}3^{\abs{w}}} \quad \textrm{as} \ n\rightarrow \infty.
\end{align*}

Now, we go on to the second factorial moment. $(\Z)_2$ counts the number of ordered pairs of distinct $[w]$-cycles. These two may or may not intersect, so we can split $(\Z)_2$ as:
\[(\Z)_2 = \Y' + \Y'', \]
where $\Y'$ counts the number of ordered pairs of vertex disjoint $[w]$-cycles, and $\Y''$ the number of ordered pairs of intersecting $[w]$-cycles.

The expectation of $\Y'$ can be computed using a similar argument as for $\E[\Z]$. As before, we write $\E[\Y']$ as:
\[\E[\Y'] = a_{n,[w]}' \cdot p_{n,[w]}' \]
where $a_{n,[w]}'$ counts the number of label·lings of an ordered pair of distinct non-intersecting $[w]$-cycles, and $p_{n,[w]}'$ is the probability that an element of $\Omega_n$ contains a given pair of disjoint sets, each of them containing $\abs{w}$ pairs of half-edges. 

In order to count $a_{n,[w]}'$ we fix again a direction and a starting vertex in each $[w]$-cycle. Thus, we will be counting $4\abs{w}^2a_{n,[w]}$. In the same way, 
\[4\abs{w}^2a_{n, w}' = 4^{2\abs{w}} n (n-1) (n-2) \ldots (n - 2\abs{w}+1),\]
and considering all the representatives of $[w]$ for each cycle, we obtain: 
\[ a_{n, [w]}' = \frac{\abs{[w]}^2}{(2\abs{w})^2}4^{2\abs{w}} n (n-1) (n-2) \ldots (n - 2\abs{w}+1).\]
As for $p_{n,[w]}'$, we observe that:
\[p_{n,[w]}' = p_{n,[w]}\cdot p_{n-\abs{w},[w]}. \]
Therefore,
\[ \E[\Y'] = a_{n,[w]}'\cdot p_{n,[w]}' \sim \bigg(\frac{\abs{[w]}}{2\abs{w}} \bigg)^2 \frac{4^{2\abs{w}} n^{2\abs{w}}}{3^{2\abs{w}}(4n)^{2\abs{w}}} \longrightarrow \bigg(\frac{\abs{[w]}}{2\abs{w}3^{\abs{w}}} \bigg)^2 \quad \textrm{as} \ n\rightarrow \infty. \]
Now let's study $\E[\Y'']$. $\Y''$ counts the pair of $[w]$-cycles that have at least one common vertex. Note that these can be seen also as a connected (by the common vertices and edges) multi-graph $P$ which has more than one cycle. Each of these $P$, then, will have more edges than vertices by construction. Expressing $\E[\Y'']$ as in the previous cases, i.e,
\[\E[\Y''] = a_{n,[w]}'' \cdot p_{n,[w]}'',\]
we observe that the number $a_{n,[w]}''$ is of the order $O(n^{\#\textrm{vertices}})$ and $p_{n,[w]}''$ depends only on the pairs of half-edges, so it is of the order $O(n^{-\#\textrm{edges}})$.  Therefore,
\[\E[\Y''] = O(n^{\#\textrm{vertices}} \cdot n^{-\#\textrm{edges}}) = O(n^{-1}).\]

Observe that this same argument works to show that asymptotically there won't be any short circuits that are not cycles in the dual graph: non-simple circuits have more edges than vertices, so the expected number of copies of them is of the order:
\[\E[\textrm{circuits}]  = O(n^{\#\textrm{vertices}} \cdot n^{-\#\textrm{edges}}) = O(n^{-1}).\]

All in all, we obtain that: 
\[\lim_{n\rightarrow \infty} \E[(\Z)_2] = \bigg(\frac{\abs{[w]}}{2\abs{w}3^{\abs{w}}} \bigg)^2 .\]
As shown in \cite[Theorem 2]{Bollobas}, the same argument applies for any factorial moment $\E[(\Z)_m]$, $m\in\N$, and for any joint factorial moment $\E\bigg[\prod_{[w]\in \mathcal{S}} (\Z)_{m_{[w]}} \bigg]$. 
In this case, we consider the number of sequences of $\sum_{[w]\in \mathcal{S}} m_{[w]}$ distinct $[w]$-cycles. We split them into the sum of intersecting and non-intersecting, and by the same reasoning as before:
\[ \lim_{n\rightarrow \infty} \E\bigg[\prod_{[w]\in \mathcal{S}} (\Z)_{m_{[w]}} \bigg] = \prod_{[w]\in \mathcal{S}} \bigg(\frac{\abs{[w]}}{2\abs{w}3^{\abs{w}}} \bigg)^{m_{[w]}} .\]
By the method of moments, this implies that the random variables $(\Z)_{[w]\in \mathcal{S}}$ converge in distribution, as $n \rightarrow \infty$ to independent Poisson random variables $Z_{[w]}$, with mean $\lambda_{[w]}= \frac{\abs{[w]}}{2\abs{w}3^{\abs{w}}}$ for all $[w]\in \mathcal{S}$. 
\end{proof}

\subsection{Proof of Theorem \ref{PoissonYn}} \label{sectionCl_Y}

Now that we have Theorem \ref{cyclesYn}, we can re-state and prove Theorem \ref{PoissonYn} as will follow next.

Recall that $W$ was the set of all words with letters in $\{S, R, L \}$ and $\{Id, \theta, \theta^2\}$, and $|w|$ denoted the length of a word $w\in W$. On the other hand, we denoted by $\mathcal{W} = W \addslash \sim$ the set of equivalence classes of words in $W$. From this, we define:
\[\mathcal{W}_{[a,b]} \coloneqq \{[w]\in \mathcal{W} : \abs{w} > 2, \ \abs{tr([w])} > 2 \ \textrm{and} \ 2\mathrm{Re}\big[\arccosh \big(\frac{tr([w])}{2}\big)\big] \in [a,b]\}.\]
This is a finite set, consequence of the following result. 

\begin{prop}\label{lgrow}
    Let $w\in W$ be any hyperbolic word formed by the letters $\{S, R, L\}$ and $\{Id, \theta, \theta^2 \}$ of word length $\abs{w}=r>0$. Then, there exists a  constant $J(r)>0$ satisfying that:
    \begin{itemize}
        \item[$\diamond$] J(r) is strictly increasing,
        \item[$\diamond$] $\frac{J(r)}{\log(r)} \rightarrow 1$ as $r\rightarrow\infty$,
    \end{itemize}
    such that the translation length of the geodesic $\gamma$ corresponding to the equivalence class of $w$ is bounded below by $J(r)$, that is, 
    \[l_{\gamma}(w) \geq J(r).\]
\end{prop} 

\begin{proof}
    Let $P$ be the plane in $\mathbb{H}^3$ spanned by the triple $\{0,i,\infty\}$, and $d(w)$ be the distance in the upper half-space of $\mathbb{H}^3$ between the planes $P$ and $w(P)$, that is,
\begin{align*}\label{distword}
    d(w) &= \min_{\substack{(x_1,y_1)\in P \\ (x_2,y_2)\in w(P)}} \{d((x_1, y_1),(x_2, y_2))\} \\ &= \min_{\substack{(x_1,y_1)\in P \\ (x_2,y_2)\in w(P)}} \bigg\{ \arccosh\bigg(1+\frac{(x_2 - x_1)^2 + (y_2 - y_1)^2}{2y_1y_2}\bigg) \bigg\}.
\end{align*}

    We start by considering some number $J(r)>0$. Then, there are two possibilities, that $d(w) \geq J(r)$ or that $d(w) < J(r)$. From the study of each case, we aim to optimise this constant, while making it a valid lower-bound in any of these two initial cases.

    As mentioned above, we have these two possible scenarios:
    \begin{itemize}[leftmargin=20px]
        \item $d(w) \geq J(r)$: We show that $l_{\gamma}(w) \geq d(w)$. For this, it is enough to see that the axis of the isometry runs through the planes spanned by the transformation. 

        To do so, we first observe the following: Let $P$ denote the plane described by $\{0,i,\infty\}$, and consider the partition $\mathbb{H}^3 \backslash P = H_1 \sqcup H_2$, where $H_2$ contains all points in $\mathbb{H}^3$ with positive first component. Then, as described in Figure \ref{planes_spanned}, for any of the 9 isometries $M\in\{S\theta^{i}, R\theta^{i}, L\theta^{i}, \ i=0,1,2\}$, we have
        \[M(H_2) \subset H_2.\]
        From this, one infers that for any isometry $w$ product of the previous matrices, we get:
        \[w(M(H_2)) \subset w(H_2).\]
        This means that every time we add a transformation M to our word, the half-space spanned by this new isometry is contained in the previous one. This tells us, in particular, that the attracting fix point of the axis in any hyperbolic transformation $w$ -sending the initial plane $P$ to $w(P)$- has to be lie inside the hyperplane $w(P)$, and the repulsive fix point then, needs to be in $H_1$. Therefore, the axis intersects both $P$ and $w(P)$, as well as all the planes spanned by all shorter words of $w$.

        Hence, by definition of the distance, it's clear that: 
        \[d(w) = d(P, w(P)) \leq l_{\gamma}(w).\]

        Since we had $d(w) \geq J(r)$, we obtain that $l_{\gamma}(w)\geq J(r)$.
        \vspace{0.2cm}
        
        \item $d(w) < J(r)$: We can re-write any word $w\in W$ as 
        \[w = S^{s_1} (R\theta)^{s_2} (L\theta^2)^{s_3} \cdot \ldots \cdot S^{s_{n-2}} (R\theta)^{s_{n-1}} (L\theta^2)^{s_n} \theta^{t},\]
        where $s_i\in\N$ for $i=1,\ldots,n$, $s_1 + \cdots + s_n = r$ and $t\in\{0,1,2\}$. We suppose then that $w$ is of this form. 
        
        Now, we have the following: any hyperbolic word of two letters spanning a plane of empty intersection with $P$ is at distance $d = \arccosh(3)$ from it. Indeed, these words can be reduced, by equivalence, to the list:
        \[ \{ SR\theta, R\theta S, R\theta L\theta^2, L\theta^2 R\theta, SL\theta^2, L\theta^2 S \}\]
        which span all possible different planes (see Figure \ref{planes_spanned}), all of them at hyperbolic distance $\arccosh(3)$.

        Since we have, from above, that for any word $w\in W$, $w=w_1 \ldots w_n$: 
        \[\quad \quad \quad w(H_2 \sqcup P) \subset w_1 \cdots w_{n-1}(H_2 \sqcup P) \subset w_1 \cdots w_{n-2}(H_2 \sqcup P) \subset \cdots \subset H_2 \sqcup P,\]
        we get that that the geodesic realising the distance from $P$ to $w(P)$ intersects the intermediary planes in some point. The segment from $P$ to these points will be then, larger or equal that the distance between the planes, by definition of distance. Therefore, any word $w$ containing at least $k$ of these words in the list -$w_1, \ldots, w_k$- disjointly, will have distance: 
        \[d(w) \geq \sum_{i=1}^{k}d(w_i) = k \arccosh(3). \]
         
        \begin{figure}[H]
            \centering
            \includegraphics[scale = 0.4]{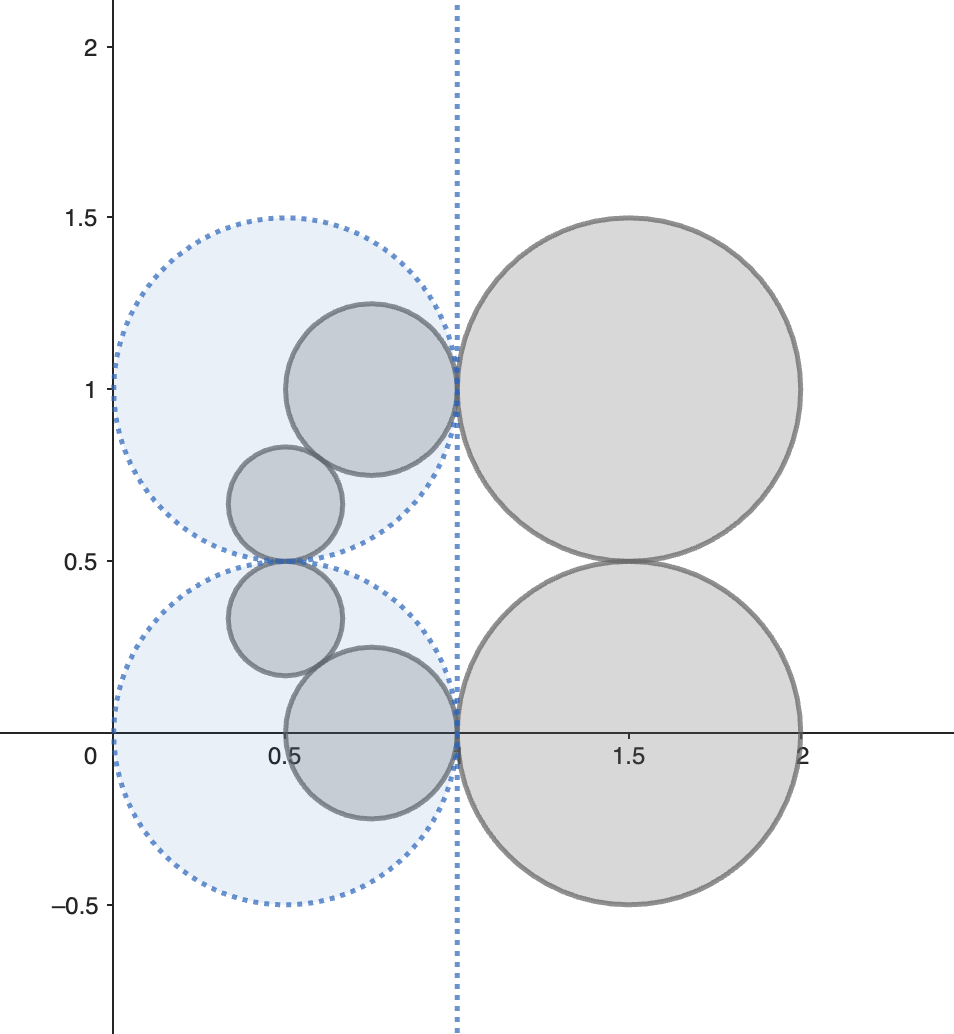}
            \caption{View from "infinity" in the upper half-space model of $\mathbb{H}^3$: the circles in blue represent the hemispheres spanned by 1-letter words, and the ones in grey the hemispheres spanned by all 2-letter hyperbolic words of positive distance from $P$.}
            \label{planes_spanned}
        \end{figure}
        Equivalently and relying it to our case, if $d(w)\leq  J(r) = k \arccosh(3)$, this tells us that $w$ contains at most $\frac{J(r)}{\arccosh(3)}$ disjoint words of the previous type. This fact gives us then an upper bound on the number of $s_i$'s in $w$:
        \[\quad \quad m = \textrm{nº$s_i$} \leq 2\cdot\textrm{nº disjoint 2-letter hyp. words}+1 \leq \frac{2J(r)}{\arccosh(3)}+1.\] 
        Since $w$ is of length $r$, by the pigeonhole principle, we obtain that there exists at least an $s_i$ such that
        \[ s_i \geq \ceil[\bigg]{\frac{r}{m}} \geq \ceil[\bigg]{\frac{r}{\frac{2J(r)}{\arccosh(3)}+1}} = \ceil[\bigg]{\frac{\arccosh(3)r}{2J(r)+\arccosh(3)}}. \]

        Let $K_r = \ceil[\bigg]{\frac{\arccosh(3)r}{2J(r)+\arccosh(3)}}$. From this last inequality, one can deduce that the word $w$ has as a sub-word one of the followings: 
        \[ \quad \quad \{S^{K_r}R\theta, \  R\theta^{K_r}S,  \ (R\theta)^{K_r}L\theta^2, \ (L\theta^2)^{K_r}R\theta, \ S^{K_r}L\theta^2, \ (L\theta)^{K_r}S\}.\]
        All of these words have distance $\arccosh(2K_r+1)$ (by direct computation). Therefore, since $l_{\gamma}(w) \geq d(w) \geq d(\tilde{w})$, for $\tilde{w}$ in the previous list, we obtain:
        \[\quad \quad l_{\gamma}(w) \geq \arccosh(2K_r+1) \geq \arccosh\bigg(2\bigg(\frac{\arccosh(3)r}{2J(r)+\arccosh(3)}\bigg)+1\bigg).\]   
    \end{itemize}
    Hence, in general, we have: 
    \[l_{\gamma}(w) \geq d(w) \geq \min \bigg\{ J(r), \arccosh\bigg(2\bigg(\frac{\arccosh(3)r}{2J(r)+\arccosh(3)}\bigg)+1\bigg)\bigg\}.\]
    This bound is optimised whenever $J(r) = \arccosh\bigg(2\bigg(\frac{\arccosh(3)r}{2J(r)+\arccosh(3)}\bigg)+1\bigg)$. Thus, we define the function 
    \begin{align*}
        J: \R_{>0} &\longrightarrow \R_{>0} \\
        r&\longmapsto J(r)
    \end{align*}
    
    implicitly by $J(r) = \arccosh\bigg(2\bigg(\frac{\arccosh(3)r}{2J(r)+\arccosh(3)}\bigg)+1\bigg)$.

    \vspace*{0.2cm}
    It rests to prove that this function is increasing: we can re-write it as 
    \begin{equation}\label{rewrite}
        (\cosh(J(r))-1)(2J(r)+\arccosh(3)) = 2\arccosh(3)r
    \end{equation}

    Since $\arccosh(3)>0$ and $r>0$, both sides of the equation are positive. Then, for $0<r_1<r_2$, we have that $\arccosh(3)r_1 < \arccosh(3)r_2$, which implies:
    \[ \frac{1}{2}(\cosh(J(r_1))-1)(2J(r_1)+\arccosh(3)) < \frac{1}{2}(\cosh(J(r_2))-1)(2J(r_2)+\arccosh(3)).\]
    
    Since both components of the product are positive, this means that either: 
    \begin{itemize}
        \item[$\diamond$] $(2J(r_1)+\arccosh(3)) < (2J(r_2)+\arccosh(3))$: yielding that $J(r_1)<J(r_2)$.
        \vspace{0.15cm}
        \item[$\diamond$] $(\cosh(J(r_1))-1) < (\cosh(J(r_2))-1)$: giving also $J(r_1)<J(r_2)$ as $\cosh(x)$ is an increasing function.
    \end{itemize}
    
    \vspace*{0.1cm}
    As a final observation, we see that $J(r)$ has a logarithmic behaviour in the limit, that is, $\frac{J(r)}{\log(r)} \rightarrow 1$ as $r\rightarrow\infty$.

    Indeed, looking at (\ref{rewrite}), we note that $J(r)$ is the $W-$function $J(r) = W_n(2\arccosh(3)r)$. In particular, since $r>0$, and both $r$ and $J(r)$ are real numbers, $J(r)$ is the principal branch, that is,
    \[J(r) = W_0(2\arccosh(3)r).\]
    
    Now, one has that, for large values or $r$, this function $W_0(r)$ is asymptotic to $\ln(r) - \ln(\ln(r)) + o(1)$. In our case, 
    
    \begin{align*}
        W_0(2\arccosh(3)r) &= \ln(2\arccosh(3)r) - \ln(\ln(2\arccosh(3)r)) + o(1) \\ 
        &= \ln(r) - \ln(2\arccosh(3)) -  \ln(\ln(2\arccosh(3)r)) + o(1).
    \end{align*}
    
  Therefore, we get
  
  \[\lim_{r\rightarrow\infty}\frac{J(r)}{\ln(r)} = \frac{\ln(r) - \ln(2\arccosh(3)) -  \ln(\ln(2\arccosh(3)r)) + o(1)}{\ln(r)} = 1.\]
\end{proof}

Note that, although this lower bound may not be be the sharpest possible, here we cannot use the growth of the traces or the distances in $\mathbb{H}^3$ spanned by the words to control the growth of the length, as it was possible for the two-dimensional case of this model (see \cite{Bram1}). Indeed, the traces are complex numbers so they don't have a natural ordering, and there exists the possibility that their absolute value decrease whenever we add a letter to a word. In the same way, it is neither true that the translation length grows whenever the distance $d(w)$ does so. For example, for $w=RLRR$, we have:
\[ d(w) = 2.63 \ \textrm{and} \  l_{\gamma}(w) = 3.47, \]
while if we add one more letter $w' = RLRRL$, we get:
\[ d(w') = 3.26 \ \textrm{and} \ l_{\gamma}(w) = 3.33. \]

\vspace*{0.1cm}
After having proved this, we get to our main point. 
For $a, b>0$, $C_{[a,b]}(Y_n)$ denoted the number of primitive closed geodesics of length $\in [a, b]$ on $Y_n$. Also, 
\[\mathcal{W}_{[a,b]} = \{[w]\in \mathcal{W} : \abs{w} > 2, \ \abs{tr([w])} > 2 \ \textrm{and} \ 2\mathrm{Re}\big[\arccosh \big(\frac{tr([w])}{2}\big)\big] \in [a,b]\}.\]

Then, we prove:
\begin{reptheorem}{PoissonYn}
For any finite collection of disjoint intervals $[a_1,b_1], \ldots, [a_t,b_t] \subset \R_{\geq 0}$ , the random vector $(C_{[a_1,b_1]}(Y_n), \ldots, C_{[a_t,b_t]}(Y_n))$
converges jointly in distribution, as $n\rightarrow\infty$, to a vector of independent random variables $$(\mathcal{C}_{[a_1,b_1]}, \ldots, \mathcal{C}_{[a_t,b_t]}),$$ where for all $i=1,\ldots,t$, $\mathcal{C}_{[a_i,b_i]}$ is Poisson distributed with parameter 
\[\lambda = \sum_{[w]\in \mathcal{W}_{[a_i,b_i]}} \lambda_{[w]}.\]
\end{reptheorem}

\begin{proof}
For simplicity, we will consider the interval $[0, l]$, and denote $C_{[0,l]}(Y_n)$ by $C_{l}(Y_n)$.
We can write $C_l(Y_n)$ as:
\[ C_l(Y_n) = \sum_{[w]\in \mathcal{W}_l} Z_{n,[w]} + \sum_{[w]\in \mathcal{W}_l} Z_{n,[w]}', \]
where $Z_{n,[w]}'$ denote the number of circuits of $G_{Y_n}$ described by $[w]\in \mathcal{W}_{[0,l]}$ that are not cycles.

Notice that the right-hand side is a sum of two finite sums of independent random variables $Z_{n,[w]}$ and $Z_{n,[w]}'$ respectively. On the one hand, since we have seen that the expected number of copies of any non-simple circuit in $G_{Y_n}$ tends to 0 as $n\rightarrow\infty$, the second summand will vanish in the limit. Indeed, we have that:

\[\lim_{n\rightarrow\infty}\E[\sum_{[w]\in \mathcal{W}_l} Z_{n,[w]}'] =  \lim_{n\rightarrow\infty} \sum_{[w]\in \mathcal{W}_l} \E[Z_{n,[w]}'] = \sum_{[w]\in \mathcal{W}_l} \lim_{n\rightarrow\infty}\E[Z_{n,[w]}']= 0.\]

For the other summand, we apply Theorem \ref{cyclesYn}. All in all, we obtain that:

\[C_l(Y_n) \xrightarrow{d} \sum_{[w]\in \mathcal{W}_l} Z_{[w]} = \mathcal{C}_l .\]
Since each $Z_{[w]}$ is Poisson distributed with parameter $\lambda_{[w]}= \frac{\abs{[w]}}{2\abs{w}3^{\abs{w}}}$, and they are all independent of each other, we obtain that the random variable $\mathcal{C}_l$ is Poisson distributed with parameter $\lambda = \sum_{[w]\in \mathcal{W}_l} \lambda_{[w]}$.
Finally, as none of the $[w]\in\mathcal{W}$ can belong to any two $\mathcal{W}_{[a,b]}$, $\mathcal{W}_{[c,d]}$ for $[a,b],[c,d]$ disjoint intervals, one concludes that for any finite collection $(C_{[a_1,b_1]}(Y_n), \ldots, C_{[a_t,b_t]}(Y_n))$, the corresponding random variables $(\mathcal{C}_{[a_1,b_1]}, \ldots, \mathcal{C}_{[a_t,b_t]})$ are independent.

\end{proof}

\section{Proof step 2: The length spectrum of $M_n$} \label{statement}

In this section, we explain first how $M_n$ can be obtained from the intermediary $Y_n$ by Dehn filling. After that, we finish the proof of Theorem \ref{Poisson} by showing that, after this compactification procedure, the result obtained in Section \ref{sectionCl_Y} still holds, that is:

\begin{reptheorem}{Poisson}
As $n\rightarrow\infty$, the length spectrum of a random compact hyperbolic 3-manifold with boundary $M_n$ converges in distribution to a Poisson point process (PPP) on $\R_{\geq 0}$ of intensity $\lambda$, where for any  $a, b\geq 0$,
\[\lambda([a,b]) = \sum_{[w]\in \mathcal{W}_{[a,b]}} \lambda_{[w]}.\]
\end{reptheorem}

\subsection{From $Y_n$ to $M_n$}

In general, a \textit{hyperbolic Dehn filling} is a standard operation in three dimensional geometry used to construct compact hyperbolic manifolds from cusped ones. It consists on the following: let $N$ be a cusped hyperbolic 3-manifold. Topollogically, we can think of $N$ as the interior of a compact manifold with toroidal boundary. Then, a \textit{Dehn filling} is the operation of gluing a solid torus $D \times S^1$ along each torus boundary component T to $N$ via a diffeomorphism $\psi: \partial D \times S^1 \rightarrow T$. This mapping takes each of the meridians $\partial D \times \{x\}$ to some closed curve $\gamma$ in T. The choice of $\gamma$ determines the Dehn filling up to homeomorphism. In our case, we do a special case of Dehn filling, that goes as follows. 

Recall that the manifolds $Y_n$ are non-compact random hyperbolic 3-manifolds with totally geodesic boundary, obtained from the gluing of octahedra described in Section \ref{gluing_oct}. It turns out, that their boundary is a random hyperbolic surface $S$ with cusps (see \cite{Bram_Jean}).

Now, if we remove a horospherical neighbourhood of each cusp (Figure \ref{cusp_nbhd}) we obtain a compact manifold with boundary consisting on $S$ together with some open cylinders. These can be seen as "half a torus"; then, by doubling this manifold, we would recover the torus, in which we could apply the exact procedure explained before. In the single one, though, we do this particular case of filling: in each cylinder, we glue a solid cylinder $D \times [0,1]$ along its boundary. Note that in this case, the image of the meridian is determined, and so the Dehn fillings as well. The result is a compact manifold M whose boundary is a compact surface $\overline{S}$.

\begin{figure}[H]
    \centering
    \includegraphics[scale=0.35]{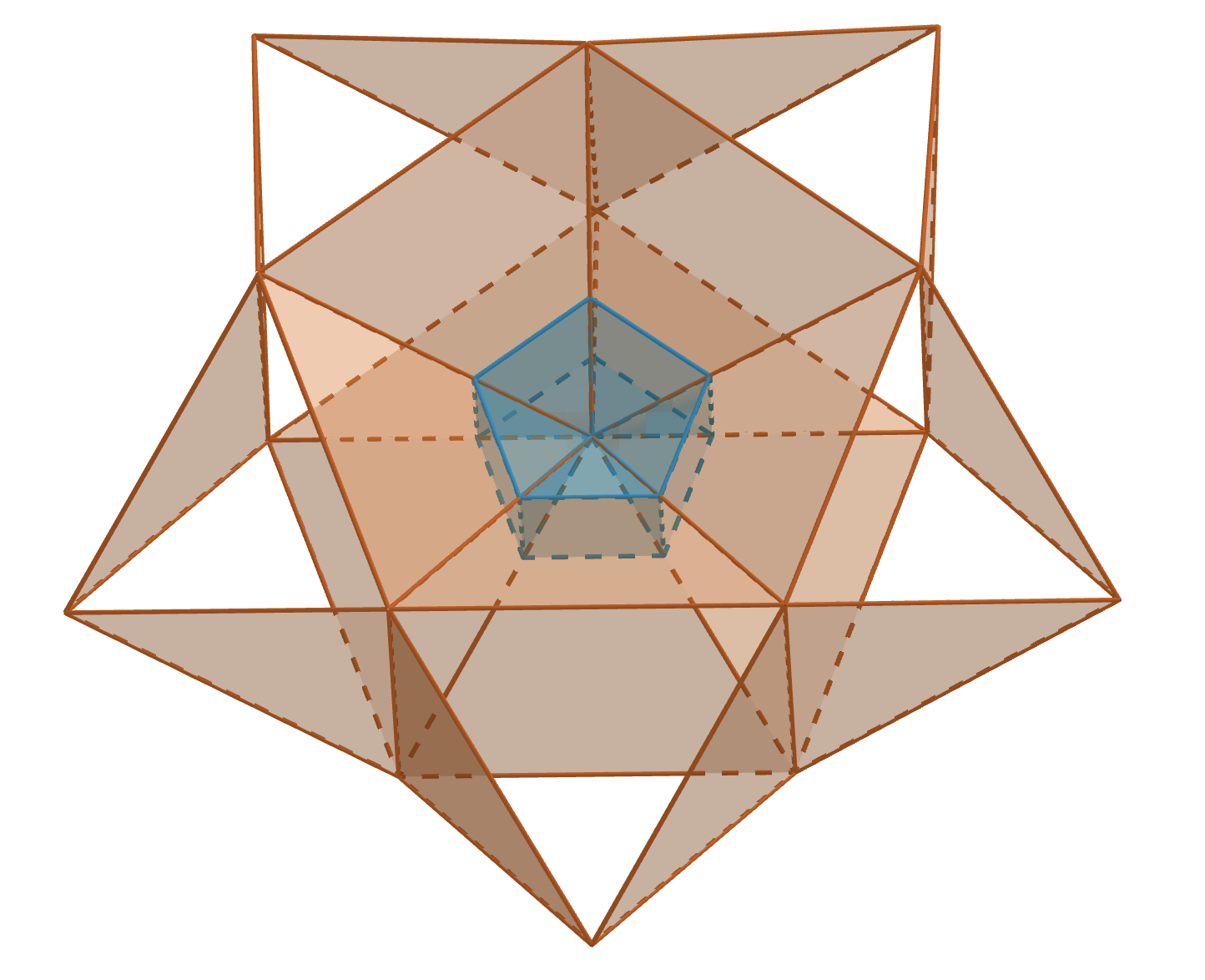}
    \caption{Cusp neighbourhood in $Y_n$ (blue).}
    \label{cusp_nbhd}
\end{figure}

The manifolds M obtained by this procedure have the same distribution as the $M_n$ presented in Section \ref{modelM_n}. Hence, we can, and will, think of $M_n$ as the Dehn fillings of $Y_n$.

\subsection{Behaviour of curves after compactification} 
In order to prove Theorem \ref{Poisson}, we need to see that the number of closed geodesics of some bounded length stays asymptotically the same after the Dehn filling. 
This comes down to checking mainly two things: first, that the length of the curves doesn't change "too much" after the compactification; and second, that they do not collapse into a point, or into each other -that is, that any two short closed geodesics that weren’t homotopic in $Y_n$, don’t become homotopic after the filling of the cusps-.

\subsubsection{Change in length of closed geodesics}

We start proving the first and main point. 

\begin{prop} \label{length_gammas}
Let $l_{max}>0$. For every $\epsilon >0$, the following holds a.a.s as $n\rightarrow \infty$: for every closed geodesic $\gamma$ in $M_n$ with length$(\gamma) \leq l_{max}$, there exists a closed geodesic $\gamma'$ in $Y_n$ such that the image of $\gamma'$ is homotopic to $\gamma$, and
\[\frac{1}{1+\epsilon}\textrm{length}(\gamma) \leq \textrm{length}(\gamma') \leq (1+\epsilon)\textrm{length}(\gamma).\]
\end{prop}

Like in \cite{Bram_Jean}, we'll control the change in geometry when doing the Dehn filling of the cusps separately in three steps. First, we will deal with "small" cusps, that is, with cusps made of few octahedra around them, and apply Andreev's theorem \cite{Andreev} to make sure the length doesn't change "too much". Then, we'll treat the "medium" and "large" cusps, in two separate steps, using a result of Futer-Purcell-Schleimer \cite[Theorem 9.30]{Futer_Purcell}. 

We set some notation beforehand, recalling as well previous objects that will appear. We start with specifying the different types of cusps. We define the combinatorial length of a cusp as the number of octahedra forming it. Then,

\vspace*{0.3cm}
\begin{tabular}{@{} *5l @{}}    \toprule
\emph{Cusps of $Y_n$} & \emph{Notation} & \emph{Description}  \\\midrule
 Small    &  $c_1, \ldots, c_s$ & Of combinatorial length up to $\frac{1}{8}\log_3(n)$  \\ 
 Medium  &  $c_{s+1}, \ldots, c_m$ & Of combinatorial length between $\frac{1}{8}\log_3(n)$ and $n^{1/4}$ \\ 
 Large & $c_{m+1}, \ldots, c_n$ & Of combinatorial length bigger than $n^{1/4}$\\\bottomrule
 \hline
\end{tabular}

\vspace*{0.3cm}
With this, we distinguish the following manifolds and their parts.
\vspace*{0.3cm}

\begin{tabular}{@{} *5l @{}}    \toprule
\emph{Models} & \emph{Description}  \\\midrule
 $Y_n$ & Hyperbolic 3-manifold constructed in section \ref{gluing_oct} conditioned on not \\
 & having loops or bigons in its dual graph\\ 
 $Y_{n,1}$ & Union of all octahedra of $Y_n$ containing a small cusp \\ 
 $Y_{n,2}$ & Complement of $Y_{n,1}$ in $Y_n$ \\\midrule
 $K_n$ & Manifold obtained from $Y_n$ by filling the small cusps\\ 
 $K_{n,1}$ & Union of all octahedra of $K_n$ containing a medium cusp\\ 
 $K_{n,2}$ & Complement of $K_{n,1}$ in $K_n$\\
 $DK_n$ & Double of $K_n$ \\\midrule
 $Z_n$ & Manifold obtained from $K_n$ by filling the medium cusps \\
 $DZ_n$ & Double of $Z_n$ \\\midrule
 $M_n$ & Manifold obtained from $Z_n$ by filling the large cusps, homeomorphic to \\
 & the $M_n$ described in Section \ref{modelM_n} \\ 
 $DM_n$ & Double of $M_n$ \\\bottomrule
 \hline
\end{tabular}

\vspace{0.3cm}
\begin{proof}[Proof of Proposition \ref{length_gammas}] 
$ $\newline

\vspace{0.15cm}

\noindent
\textit{Filling small cusps}.
Here we use, more precisely, a consequence of Andreev's theorem (for the original theorem, see \cite{Andreev}). As stated in \cite{Bram_Jean}, this gives the following result:

\begin{lemma}[\cite{Bram_Jean},  Lemma 3.6] \label{lemma_Andreev}
There exists $J_0>0$ such that the following holds for any $\epsilon >0$ below the Margulis constant for $\mathbb{H}^3$, and any $\delta>0$: with probability at least $1-\delta$ in the model $Y_n$ for $n$ large enough, we have:
\begin{itemize}
    \item[1.] The $\epsilon$-thick part of the image of $Y_{n,1}$ in $K_n$ is $J_0$-bilipschitz to that of $Y_{n,1}$,
    \item[2.] The image of $Y_{n, 2}$ in $K_n$ is isometric to $Y_{n, 2}$.
\end{itemize}
\end{lemma}
We can think of $K_n \subset Z_n \subset M_n$, hence the same lemma holds in particular for the manifold $K_n$. 

\vspace{0.2cm}
So, let $\gamma$ be a closed geodesic in $K_n$. We look at its preimage $\tilde{\gamma}$ in $Y_n$. We have the following cases:

\begin{itemize}[leftmargin=*]
\item The curve $\tilde{\gamma}$ lies entirely in $Y_{n, 2}$: by the second point of Lemma \ref{lemma_Andreev}, its length will be exactly the same as the one of $\gamma$. Thus, we can take this as $\gamma'$.
\item  The curve $\tilde{\gamma}$ lies partly in $Y_{n, 1}$: To handle this case, we recall the notion of tangle-freeness for graphs. 

Let G be a multigraph. A neighbourhood of radius $l$ in G is the subgraph spanned by vertices at graphical distance at most $l$ from some fixed vertex. Then, 
\begin{defi}
A multigraph G is tangle-free if it contains at most one cycle (including loops and multiple edges). G is $l$-tangle-free if every neighbourhood of radius $l$ in H contains at most one cycle. Otherwise, we say that G is tangled or $l$-tangled.
\end{defi}
         
Now, it turns out that random 4-regular graphs -so in particular $G_{Y_n}$- are $l$-tangle free, for $l>0$ not too large. More precisely, 
\begin{lemma} [\cite{Tangle-free}, Lemma 9] \label{tanglefree}
Let $d\geq 3$, and $G_d$ be a random $d$-regular graph on $n$ vertices generated by the configuration model. Then, 
\[ \mathbb{P} [\ \textrm{$G_d$ is $l$-tangle-free}\ ] = 1 - O\Big(\frac{(d-1)^{4l}}{n}\Big).\]
\end{lemma}

Using this property, we argue as follows: we have that the preimage of $\gamma$, $\tilde{\gamma}$, homotops to a cycle in the dual graph $G_{Y_n}$. As the curve lies in some part of $Y_{n, 1}$, its cycle intersects with the one representing the parabolic element that goes around a small cusp in $Y_{n, 1}$.  On the other hand, since $l(\gamma)<l_{max}$, by Proposition \ref{lgrow} we deduce that the cycle in $G_{Y_n}$ corresponding to $\tilde{\gamma}$ has length bounded above by $K + \frac{1}{8}\log_3(n)$, where $K>0$.
Therefore, both cycles are inside a neighbourhood of radius $l\leq \frac{1}{8}\log_3(n)$ in $G_{Y_n}$.

Nevertheless, we know that the dual graph $G_{Y_n}$ is $l$-tangle free for $l<\frac{1}{4}\log_3(n)$ (Lemma \ref{tanglefree}), which means that as $n\rightarrow\infty$, there won't be more than one cycle in any neighbourhood of this radius w.h.p. This tells us, therefore, that a.a.s this case will not occur.

\end{itemize}

\vspace{0.2cm}
\noindent
\textit{Filling medium cusps}. 
Now, let $\gamma$ be a closed geodesic in $Z_n$. A priori Lemma \ref{lemma_Andreev} can also be applied. So, as before, we consider its preimage $\tilde{\gamma}$ in $K_n$, and treat two possible scenarios:
\begin{itemize}[leftmargin=*]
\item  The curve $\tilde{\gamma}$ lies entirely in $K_{n, 2}$: by the second point of Lemma \ref{lemma_Andreev}, its length will be exactly the same as the one of $\gamma$.

\item The curve $\tilde{\gamma}$ lies partly in $K_{n, 1}$: Lemma \ref{lemma_Andreev} tells us that the part lying in $K_{n, 2}$ will remain exactly of the same length as in $\gamma$. On the other hand, we also have that the $\epsilon$-thick part of the image of $K_{n,1}$ in $Z_n$ is $J_0$-bilipschitz to that of $K_{n,1}$, for some $J_0>0$. However, as one gets very close to the cusp, this bilipschitz constant may degenerate, so we cannot get explicit control on the change in length.
\end{itemize}
Therefore, to treat this case, we'll use the result of Futer-Purcell-Schleimer stated next.

\begin{defi}
Let N be a hyperbolic 3-manifold with rank-two cusps $C_1, \ldots, C_n$. Consider a slope $s_j$ for each cusp torus $\partial C_j$. Then, the normalized length of $s_j$ is 
\[ L_j = L(s_j) = \frac{len(s_j)}{\sqrt{area(\partial C_j)}}, \]
where $len(s_j)$ is the length of a geodesic representative of $s_j$ on $\partial C_j$.
                
Let $s = (s_1, \ldots, s_n)$ be the vector of all slopes. We define the total normalized length $L = L(s)$ via the formula
\[\frac{1}{L^2} = \sum_{j=1}^{n} \frac{1}{L_j^2}.\]
\end{defi}

Then, by Futer-Purcell-Schleimer, we have the following:
\begin{theorem}[\cite{Futer_Purcell}, Theorem 9.30] \label{purcell}
Fix $0<\epsilon \leq \log{3}$. Let $M$ be a finite-volume hyperbolic 3-manifold and $\Sigma$ a geodesic link in $M$. Let $N = M - \Sigma$. Suppose that in the complete structure on $N = M - \Sigma$, the total length of the meridians of $\Sigma$ satisfies:
\[ L^2 \geq \frac{2\pi \cdot 6771\cosh{(0.6\epsilon + 0.1475)}^5}{ \epsilon^5} + 11.7.\]

\vspace{0.1cm}
Then there is a cone-deformation $M_t$ connecting the complete hyperbolic metric $g_0$ on $N$ to the complete hyperbolic metric $g_{4\pi^2}$ on $N$. Furthermore, the cone deformation gives a natural identity map $id:(M - \Sigma,g_0) \rightarrow (M - \Sigma,g_{4\pi^2})$, such that $id$ and $id^{-1}$ restrict to:
\[ id\restriction_{N^{\geq \epsilon}} : N^{\geq \epsilon} \xhookrightarrow{} M^{\geq \frac{\epsilon}{1.2}}, \quad  id\restriction_{M^{\geq \epsilon}}: M^{\geq \epsilon} \xhookrightarrow{} N^{\geq \frac{\epsilon}{1.2}},\]
that are $J$-bilipschitz inclusions for 
\[ J = exp\bigg(\frac{11.35 l}{e^{5/2}}\bigg) \quad \textrm{and} \quad l \leq \frac{2\pi}{L^2 - 11.7}. \]
\end{theorem}

For convenience of the following arguments, we consider as our cusped manifold the double of the manifold $K_n$, $DK_n$, which is a non-compact manifold without boundary. Then, cusps neighbourhoods correspond to (thickened) tori, which get replaced by solid tori after compactification. 
We want to do this filling with Margulis tubes. We recall that a Margulis tube of radius $r>0$ is a tubular neighbourhood of a closed geodesic $\alpha$. Its radius $r$ denotes the distance between the core geodesic of the tube $\alpha$ and its boundary. In $\mathbb{H}^3$, it's an infinite solid cylinder, while in a hyperbolic 3-manifold, it's diffeomorphic to a solid torus. We'll denote it by $T_r(\alpha)$. Although a priori they might not glue very nicely, Theorem \ref{purcell} assures enough control on the change in geometry, so that the subsequent model is a good representation for the geometry of the hyperbolic metric in $DZ_n$.

Another important comment is that, when considering the doubles $DK_n$, $DZ_n$ new closed geodesics appear. However, our asymptotics only concern closed geodesics lying in one of the copies of the single manifolds, so we will only consider these.
                
As a final remark, note that in order to get sharper bounds, we will compactify the medium cusps individually one by one, and apply the previous result at each step, that is, with $L = L_j$, for $j=s+1, \ldots, m$. Although that could make the bilipschitz constants accumulate, this will not occur as the medium cusps don't intersect a.a.s. The proof of this claim is at the end of the section (Claim \ref{claim}).

So, let $\epsilon>0$ be some number arbitrarily small.
For the compactification of each medium cusp, we take a Margulis tube $T_{r(\epsilon)}(\alpha)$ which separates the $\epsilon$-thin and $\epsilon$-thick part of $DZ_n$. In this way, after the filling of all of them, we get that $DZ_n^{<\epsilon} = \sqcup_{j=1}^{m} T_{r(\epsilon)}(\alpha_j)$, corresponding to small and medium cusps.

Note that the only closed geodesics that lie on $DZ_n^{<\epsilon}$ are the core geodesics $\alpha_j$. Indeed, these Margulis tubes are solid tori, so their fundamental groups are isomorphic to $\mathbb{Z}$. In each of them, then, there is only one free-homotopy class of closed curves -form by the core curve and its powers- in which the $\alpha_j$ is the geodesic representative. On the other hand, we must recall that we are interested only in closed geodesics lying in one copy of $Z_n$, in which these core geodesics don't exist. Therefore, they can be dismissed for our argument. 

This yields, then, that the only geodesics object of study are either the ones lying in the $\epsilon$-thick part of $DZ_n$, or the ones having a part inside some Margulis tube corresponding to a medium cusp.

For the first case, we can use directly Theorem \ref{purcell}. This gives us a bilipschitz inclusion between $DZ_n^{\geq \epsilon}$ and $DK_n^{ \geq \frac{\epsilon}{1.2}}$, from which we obtain that the length of the preimage $\tilde{\gamma} \in DK_n$ of the closed geodesic $\gamma \in DZ_n$ might increase at most by a factor $J_{\epsilon, L}\in (1,1.0005)$. Thus, if take as $\gamma'$ the geodesic in the homotopy class of $\tilde{\gamma}$ in $DK_n$, we obtain:
\[length(\gamma') \leq J_{\epsilon, L} \cdot length(\gamma).\]

On the other hand, if we measure now the length of $\textrm{Im}(\gamma')$ in the metric $DZ_n$, we have, by this same result, that its length may decrease by at most $\frac{1}{J_{\epsilon, L}}$. 
As $\gamma$ is the geodesic in the homotopy class of $\textrm{Im}(\gamma')$ in $DZ_n$, we obtain:
\[\frac{1}{J_{\epsilon, L}} \cdot length(\gamma) \leq length(\gamma'),\]
where, in both inequalities, $J_{\epsilon, L}\rightarrow 1$ as $n\rightarrow\infty$.

Finally, let's study the second case. Let $\gamma\in DZ_n$ such that it lies partly in some Margulis tube $T_{r(\epsilon)}(\alpha_j)$, corresponding to a medium cusp. We will see that this curve cannot be in the "very thin" part of $DZ_n$.

\begin{figure}
    \centering
    \includegraphics[scale=0.35]{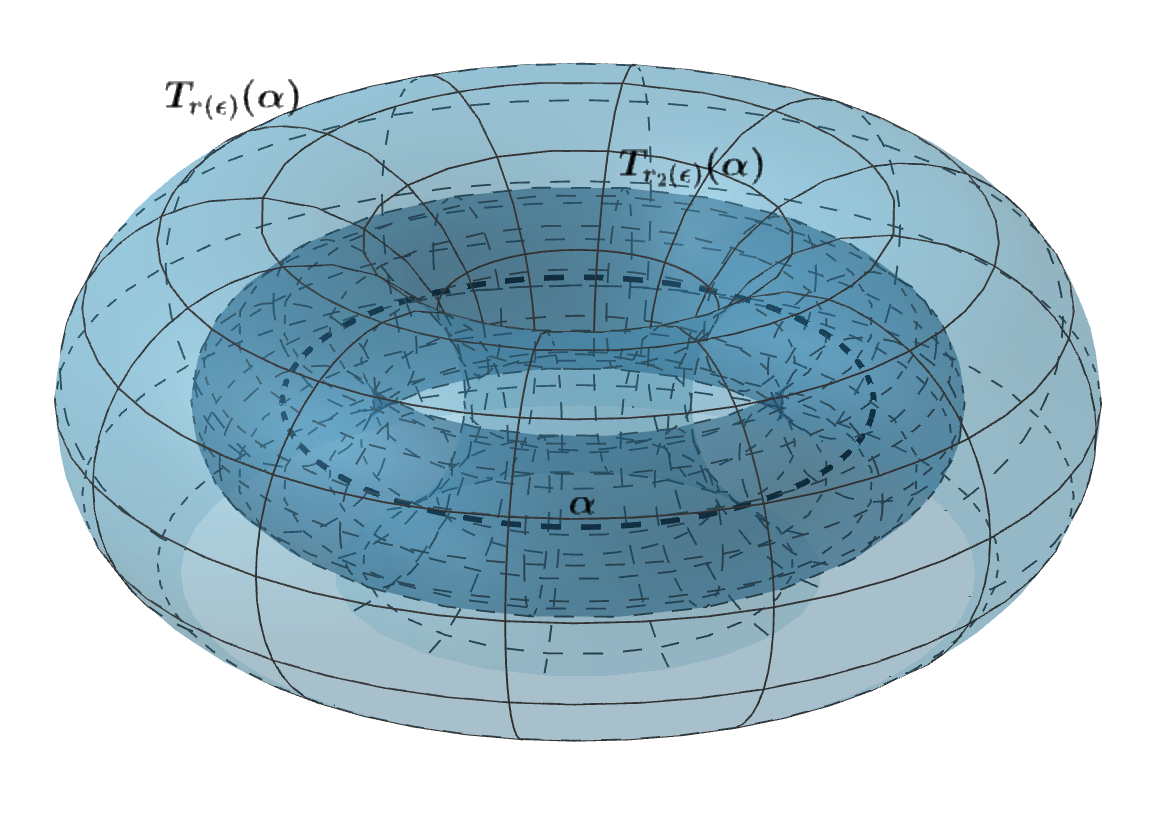}
    \caption{Nested Margulis tubes of radius $r(\epsilon$) and $r_2(\epsilon)$ around the core geodesic $\alpha$.}
    \label{Margulistubes}
\end{figure}

For this, we take $0<\delta<\epsilon$ such that:
\[ \arccosh\bigg(\frac{\epsilon}{\sqrt{7.256 \delta}}\bigg)-0.0424 > l_{max}.\]
We consider also the nested Margulis tube $T_{r(\delta)}(\alpha_j)$ of radius $r(\delta)>0$ around the core curve $\alpha_j$, which contains the $\delta$-thin part of $DZ_n$ around this core geodesic.

Now, suppose that $\gamma$ enters the $\delta$-thin part of that manifold. Since it lies only partly in the $\delta$-thin part on $DZ_n$, that means it has to exit both $T_{r(\delta)}(\alpha_j)$ and $T_{r(\epsilon)}(\alpha_j)$ at some point. Therefore, its length needs to be larger than twice the distance between the boundaries of the two tubes. 

Theorem 1.1 from Futer-Purcell-Schleimer \cite{DistMargulistube} gives lower and upper bounds on the distance between these two nested Margulis tubes, provided that the length of the core geodesic is smaller than $\delta$. Thus, in order to be able to apply it, we first check this condition. Using \cite[Corollary 6.13]{Futer_Purcell}, we have that the length of the core curve $\alpha_j$ is bounded by: 
\[ l(\alpha_j) < \frac{2\pi}{L^2 - 28.78},\]
where $L$ is the total normalized length of the compactified cusp. Since here we're dealing with medium cusps, we have that $L \geq \frac{1}{8}\log_3(n)$. Hence, 
\[ l(\alpha_j) \leq \frac{2\pi}{\frac{1}{8}\log_3(n)^2 - 28.78} < \frac{16\pi}{\log_3(n)^2}, \]
which is less that $\delta$ for $n$ large enough. Therefore, we can apply \cite[Theorem 1.1]{DistMargulistube}, that yields: 
\[ d(\partial T_{r(\delta)}(\alpha), \partial T_{r(\epsilon)}(\alpha)) \geq \arccosh\bigg(\frac{\epsilon}{\sqrt{7.256 \delta}}\bigg)-0.0424 > l_{max}.\]

This leads to a contradiction, since for $\gamma$ to enter into any of the smallest tubes $T_{r(\delta)}(\alpha_j)$, for $j=s+1, \ldots, m$, its length needed to be larger than twice that distance, yet the length of the curve $\gamma$ is bounded by $l_{max}$. Therefore, we conclude that asymptotically as $n\rightarrow\infty$, $\gamma$ won't enter the $\delta$-thin part of $DZ_n$. 

We can apply, consequently, the bilipschitz equivalences of Theorem \ref{purcell} to the $\delta$-thick part of $DZ_n$. In the same way as before, by taking as $\gamma'$ the geodesic in the homotopy class of the preimage of $\gamma$ in $DK_n$, we get:
\[ \frac{1}{J_{\delta, L}} \cdot length(\gamma) \leq length(\gamma') \leq J_{\delta, L} \cdot length(\gamma),\]
where $J_{\delta, L} \rightarrow 1$ as $n\rightarrow \infty$. 

After the filling of all medium cusps, we obtain then that the length of $\gamma'$ will be bounded by the length of $\gamma$ times some product of bilipschitz constants, all tending to 1 as $n\rightarrow \infty$.

\vspace{0.3cm}
\noindent
\textit{Filling large cusps}. Finally for this step, we rely entirely on the previously mentioned result of Futer-Purcell-Schleimer. 

The structure of the argument is exactly like in the case for medium cusps. The reason, then, for not treating them at the same time is that, in that case, the conditions needed to apply Theorem \ref{purcell} wouldn't be satisfied. Indeed, if we considered all cusps of combinatorial length $\geq \frac{1}{8}\log_3(n)$, we would have that the total normalised length would be smaller than the lower bound specified in the statement.

So, as before, we take as our cusped manifold the double of $Z_n$, $DZ_n$. Equally, the compactification process consists of taking out a horospherical neighbourhood of each cusp, 
and filling it by gluing a Margulis tube. This time, however, the procedure is done for all cusps at once, as we don't have the certainty that large cusps don't intersect each other. Then, the resulting manifold, made of the $\epsilon$-thick part of $DZ_n$, for some $\epsilon>0$ -which is compact- and the Margulis tubes attached along the boundary, is a reasonable model for the compact manifold $DM_n$.

The two same cases for the position of $\gamma$ in $DM_n$ appear, and following the same argument, we get that there is a closed geodesic $\gamma'$ in $DZ_n$ such that their lengths differ by a multiplicative constant that tends to 1 as $n\rightarrow\infty$.

\vspace{0.3cm}
All in all, we obtain that for any closed geodesic $\gamma$ in $M_n$ of uniformly bounded length, we can find a closed geodesic $\gamma'$ in $Y_n$ such that its length varies from the one of $\gamma$ by at most some product of bilipschitz constants that tend to 1 as $n\rightarrow\infty$. Hence, this gives Proposition \ref{length_gammas}.

\end{proof}

Before moving into the next point, let's prove the remaining claim, stated before:
\begin{claim} \label{claim}
    Medium cusps don't intersect each other a.a.s. Consequently, when compactifying them one by one, the bilipschitz constants  won't accumulate.
\end{claim}
\begin{proof}
Theorem 2.4 (d) from \cite{Bram_Jean} claims that, for any $K,L = o(n^{1/3})$, the expected number of pairs of edges of size $\leq K$ and $\leq L$ that are incident to a common tetrahedron in the model $M_n$ tends to zero, as $n\rightarrow \infty$. In this setting, having an edge of size $k$ implies having a cusp of length $k$, as there are $k$ glued tetrahedra around it. Thus, the result tells us that the expected number of pairs of intersecting cusps of lengths $\leq C = o(n^{1/3})$ is asymptotically negligible.
                
This fact, in particular, assures that the bilipschitz constants resulting at each step will not accumulate for the following: when applying Theorem \ref{purcell} to each medium cusp, the bilipschitz equivalence will apply only in the corresponding part of $K_{n_1}$, which is not incident to any other. Then, outside each compactified cusp (i.e in $K_{n_2}$)
the isometry given by Lemma \ref{lemma_Andreev} will still hold. 
\end{proof}

\subsubsection{Non-homotopy of closed geodesics}

Once we have this equivalence on the lengths, it rests to rule out the only setback that could arise: the possibility that closed geodesics after compactification are homotopically trivial, or get homotoped into one another. Hence, we prove now our second main point.
\begin{lemma} \label{homotopy}
Let $l_{max}>0$. A.a.s as $n\rightarrow\infty$, the images in $M_n$ of any two non-homotopic closed geodesics in $Y_n$ -of lengths bounded by $l_{max}$- are also non-homotopic, neither homotopically trivial.
\end{lemma}

\begin{proof}

We show first that, a.a.s, the image in $M_n$ of a closed geodesic in $Y_n$ is not homotopically trivial. Observe that for a short geodesic to be homotopically trivial after the compactification, it needs to go around at least two small cusps in $Y_n$. Indeed, if a geodesic homotopic to a point in $M_n$ goes around a single cusp, this would imply that, pre-compactification, it is homotopic to the cusp. But since it is a hyperbolic element, this cannot occur. On the other hand, both the cusps and the distance between them have to be of length $<l_{max}$, since the geodesic is so. 

If we look, then, at the path in the dual graph of $Y_n$ that this potentially homotopically trivial geodesic does, we see that it is a concatenation of cycles at uniformly bounded distance. Having this would imply that $G_{Y_n}$ is $l$-tangled, for $l\leq\frac{1}{8}\log_3(n)$. However, we know by \cite[Lemma 9]{Tangle-free} that this probability tends to 0 as $n\rightarrow\infty$. Therefore, we can conclude that a.a.s as $n \rightarrow \infty$, closed geodesics of this bounded length don't become homotopically trivial after the compactification.

Now, to treat the case of geodesics becomic homotopic to one another, we'll consider the doubles of the manifolds $DY_n$ and $DM_n$. As before, we argue by contradiction. Suppose there exist a set $A_n\in\Omega_n$ with $\lim_{n\rightarrow\infty}\mathbb{P}(A_n)>0$, for which the model $DY_n$ verifies the following: there exist two non-homotopic closed geodesics $\gamma_1$ and $\gamma_2$ in $DY_n$, such that their images $\phi(\gamma_1)$ and $\phi(\gamma_2)$ in $DM_n$ are homotopic to some other closed geodesic $\gamma$, of smaller length $l>0$ (or analogously, that one of them is homotopic to the other).

We take, then, the cover of $DM_n$ corresponding to that geodesic, that is, we consider:
\[\mathbb{H}^3 / \langle \gamma \rangle,\] 
where $\langle \gamma \rangle$ is an infinite cyclic subgroup of the fundamental group of $DM_n$, generated by the loxodromic transformation $\gamma$. This is diffeomorphic to $S^1 \times \mathbb{R}^2$, that is, a solid torus (see Figure \ref{cover}).

\begin{figure}[H]
    \centering
    \includegraphics[scale=0.4]{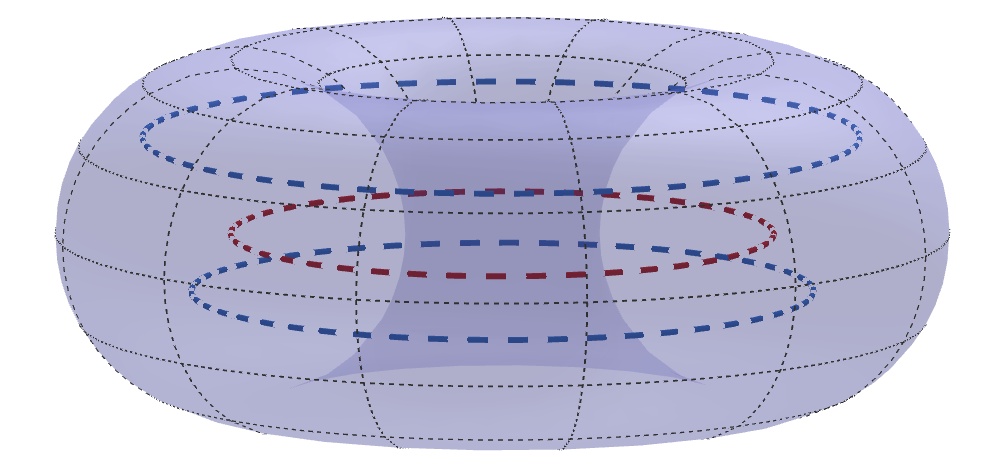}
    \caption{The curves $\tilde{\gamma}$ (red) and $\widetilde{\phi(\gamma_1)}$, $\widetilde{\phi(\gamma_2)}$ (blue) in the cover $\mathbb{H}^3 / \langle \gamma \rangle$.}
    \label{cover}
\end{figure}

We lift now the two curves $\phi(\gamma_1)$ and $\phi(\gamma_2)$ to the cover. Since $\phi(\gamma_1)$ and $\phi(\gamma_2)$ are homotopic to $\gamma$, their lifts $\widetilde{\phi(\gamma_1)}$ and $\widetilde{\phi(\gamma_2)}$ in this cover will also contract into $\tilde{\gamma}$. 

Once in this setting, we can use the fact that the geometry of this cover "flares out": the radius of the solid torus grows exponentially fast. On the other hand, we know, by Proposition \ref{length_gammas}, that the lengths of the images are very close to those in $DY_n$. Hence, the lengths of the lifts, although may increase, will continue to be finite.

All this implies that these curves cannot be too far from $\tilde{\gamma}$, or at least, they need to be at bounded distance from it. Thus, if we denote by $D_1$ the distance from $\tilde{\gamma}$ to $\widetilde{\phi(\gamma_1)}$, and $D_2$ the distance from $\tilde{\gamma}$ to $\widetilde{\phi(\gamma_2)}$, we have:
\[ d(\widetilde{\phi(\gamma_2)}, \widetilde{\phi(\gamma_2)}) \leq d(\widetilde{\phi(\gamma_1)}, \tilde{\gamma}) + d(\tilde{\gamma}, \widetilde{\phi(\gamma_2)}) \leq D_1 + D_2 = D. \]

We have obtained, then, that in the cover, the two curves $\widetilde{\phi(\gamma_1)}$ and $\widetilde{\phi(\gamma_2)}$ are at bounded distance $D$ from each other. But since distances in $DM_n$ are smaller or equal that distances in $\mathbb{H}^3 / \langle \gamma \rangle$, this yields that $\phi(\gamma_1)$ and $\phi(\gamma_2)$ are also at bounded distance in $DM_n$. And in particular, so are their paths in the dual graph $G_{DM_n}$.

However, drawing on graph theory tools, we have that in a random regular graph, the probability that there are two closed paths of lengths $l_1$ and $l_2$ at bounded distance $d>0$ apart, tends to 0 as the number of vertices goes to infinity \cite[Lemma 5.5]{BrooksMakover}. This, then, leads to a contradiction, proving that a.a.s this homotopy will not occur.

\end{proof}

\subsection{Theorem \ref{Poisson}}

Finally, with both Proposition \ref{length_gammas} and Lemma \ref{homotopy}, we are ready to prove our main result. For simplicity of the argument, we will use the same equivalent reformulation of the theorem as for Theorem \ref{PoissonYn}.

\begin{reptheorem}{Poisson}
For any finite collection of disjoint intervals $[a_1,b_1], \ldots, [a_t,b_t] \subset \R_{\geq 0}$, the random vector $(C_{[a_1,b_1]}(M_n), \ldots, C_{[a_t,b_t]}(M_n))$
converges jointly in distribution, as $n\rightarrow\infty$, to a vector of independent random variables $$(\mathcal{C}_{[a_1,b_1]}, \ldots, \mathcal{C}_{[a_t,b_t]}),$$ where for all $i=1,\ldots,t$, $\mathcal{C}_{[a_i,b_i]}$ is Poisson distributed with parameter 
\[\lambda = \sum_{[w]\in \mathcal{W}_{[a_i,b_i]}} \lambda_{[w]}.\]
\end{reptheorem}

\begin{proof}
Proposition \ref{length_gammas} states that their lengths in $Y_n$ and $M_n$ are comparable, for $n$ large enough. This, together with Lemma \ref{homotopy} yields that the number of short closed geodesics in $Y_n$ stays asymptotically the same after compactification. More precisely, given any $a,b\geq0$, the bounds obtained in the same Proposition \ref{length_gammas} imply that, for every $\epsilon>0$,
\[C_{[(1+\epsilon)a, \frac{b}{(1+\epsilon)}]}(Y_n) \leq C_{[a,b]}(M_n) \leq C_{[\frac{a}{(1+\epsilon)}, (1+\epsilon)b]}(Y_n).\]

Since by Theorem \ref{PoissonYn}, $C_{[a,b]}(Y_n)$ converge to a Poisson distributed random variable of parameter $\lambda = \sum_{[w]\in \mathcal{W}_{[a,b]}} \lambda_{[w]}$ as $n\rightarrow \infty$, we can conclude from the inequality above that the same asymptotic distribution of closed geodesics holds for the compact manifold $M_n$.
Moreover, again as a consequence of Theorem \ref{PoissonYn}, for any finite collection of disjoint intervals, these limiting random variables are independent. 
\end{proof}

\printbibliography

\end{document}